\documentclass[a4paper,10pt]{article}

\usepackage[margin=3cm,footskip=1cm]{geometry}

\usepackage{amsmath,amssymb,amsthm,bm,mathtools,xspace,booktabs,url}
\usepackage{graphicx,etoolbox}
\usepackage{ dsfont }
\usepackage{hyperref}
\usepackage{bbold}
\usepackage{tikz}
\usepackage[shortlabels]{enumitem}
\usepackage{xcolor} 
\usepackage{mathrsfs} 
\usepackage{color}

\newcommand{\dd}{\mathrm{d}}

\newcommand{\ME}{\mathcal{E}}
\newcommand{\MS}{\mathcal{S}}

\newcommand{\thz}{\theta}
\newcommand{\thh}{\widehat{\theta}}
\newcommand{\thg}{\widetilde{\theta}}
\newcommand{\Thz}{\Theta}
\newcommand{\Thh}{\widehat{\Theta}}
\newcommand{\Thg}{\widetilde{\Theta}}

\makeatletter
\global\let\tikz@ensure@dollar@catcode=\relax
\makeatother
\setlist{
  listparindent=\parindent,
  parsep=0pt,
}
\usepackage{amssymb}

\usepackage{cleveref}

\AtBeginEnvironment{thebibliography}{
  \small
  \setlength\itemsep{0.75em plus 0.5em minus 0.75em}%
}

\usepackage{caption}
\usepackage[raggedright]{titlesec}

\numberwithin{equation}{section}
\usepackage{amsfonts}

\theoremstyle{plain} 
\newtheorem{theorem}{Theorem}[section]

\newtheorem{Corollary}[theorem]{Corollary}

\newtheorem{definition}[theorem]{Definition}

\theoremstyle{definition} 

\newtheorem{Remark}[theorem]{Remark}
 {
      \theoremstyle{plain}
      
  }

\usepackage{authblk}

\usepackage{tasks}

\date{}

\makeatletter
\newcommand\CorrespondingAuthor[1]{
  \begingroup
  \def\@makefnmark{}
  \footnotetext{Corresponding author: #1}
  \endgroup
}
\makeatother

\makeatletter
\renewenvironment{abstract}{%
  \small%
  \begin{center}%
    \bfseries \abstractname\vspace{-.5em}\vspace{\z@}%
  \end{center}%
  \quote%
}
{
\endquote}
\makeatother

\usepackage[above,below,section]{placeins}

\usepackage[all]{onlyamsmath}

\usepackage{multirow}

\usepackage{xcolor}
\usepackage{color}

\usepackage{mathrsfs}
\usepackage{pdfpages}

\definecolor{darkmagenta}{rgb}{0.5,0,0.5}
\definecolor{darkgreen}{rgb}{0,0.6,0}
\definecolor{darkblue}{rgb}{0,0,0.6}
\definecolor{darkred}{rgb}{0.8,0,0}
\definecolor{mellow}{rgb}{.847, 0.72, 0.525}

\begin{document}

\title{Ruin-dependent bivariate stochastic fluid processes}

\author{
Hamed Amini \thanks{Department of Industrial and Systems Engineering, University of Florida, Gainesville, FL, USA, email:  aminil@ufl.edu} \ \ \ 
Andreea Minca\thanks{Cornell University, School of Operations Research and Information Engineering, Ithaca, NY, 14850, USA, email: {\tt acm299@cornell.edu}}\ \ \ 
Oscar Peralta\thanks{Cornell University, School of Operations Research and Information Engineering, Ithaca, NY, 14850, USA, email: {\tt op65@cornell.edu}}
}

\maketitle

\begin{abstract}
This paper presents a novel model for bivariate stochastic fluid processes that incorporate a ruin-dependent behavioral switch. Unlike typical models that assume a shared underlying process, our model allows each process to operate independently until a ruin event in one triggers a change in the other. We develop a mathematical framework for our model, exploring its properties and providing closed-form expressions for approximations of key performance metrics, particularly the joint law of the ruin times. Our approach introduces a class of compatible pathwise approximations to analyze ruin probabilities, which we subsequently study through a matrix-analytic framework.
\bigskip


\end{abstract}

\section{Introduction}

Stochastic fluid processes have been widely used to model various real-world systems, such as dam theory \cite{loynes1962continuous}, telecommunications \cite{anick1982stochastic}, queuing systems \cite{rogers1994fluid,asmussen1995stationary,karandikar1995second}, and insurance surplus processes \cite{badescu2005risk,bladt2019parisian}. These models are often employed to study the dynamic behavior of systems in which resources are continuously depleted and replenished over time. This behavior is dictated by an underlying time-homogeneous Markov jump process. Although one-dimensional processes are well understood (see e.g., \cite{latouche2018analysis} and references therein), the development and analysis of their multivariate counterparts present a significant challenge due to the need to account for complex interdependencies  between the various processes. Such dependence plays a crucial role in these systems as it greatly impacts their behavior and, consequently, the outcomes of any related decision-making processes

A bivariate stochastic fluid process comprises of two continuous-time processes. It is assumed that the fluid levels of each process fluctuate due to random events. The existing literature on bivariate stochastic fluid processes (see e.g., \cite{rabehasaina2006moments,bean2013stochastic,o2017stationary}) has primarily focused on the development of models that share a common underlying process. However, in many real-world applications, this assumption does not apply. Instead, each process evolves independently until a specific event related to one coordinate triggers a shift in behavior for the other. In this paper, we introduce a novel bivariate stochastic fluid process model that incorporates a ruin-dependent behavioral switch, which is triggered upon the ruin of one of the processes. In this model, the bivariate process becomes  Markov jump process with randomized intensities.

In the context of a stochastic fluid process, `ruin' refers to a situation where the fluid level of a process drops below level $0$, thereby rendering the process unsustainable. In the model proposed here, the interdependence between the two processes alters when one process experiences ruin, signifying a behavioral switch. This ruin-dependent behavioral switch can be interpreted as a shift in operating conditions or decision-making rules within a real-world system. For example, within the realm of insurance risk management, the ruin of one portfolio might prompt a shift towards a more conservative investment strategy for the remaining portfolio.

In \cite{aalto2000tandem,kroese2001joint}, the authors investigate the joint stationary distribution of on-off tandem fluid queues, i.e., stochastic fluid processes with two underlying states, where one of the components changes behavior whenever the other queue is idle (i.e., at level $0$). Through operator theoretic considerations, this work was later extended in \cite{o2017stationary} to include the case with finite underlying states. In \cite{rabehasaina2006moments}, the author computes the joint moments for the stationary distribution of a network of stochastic fluid processes whose inflow/outflow fluid dynamics between neighboring states is proportional to their queue size. In \cite{bean2013stochastic}, the authors consider the first return probability for the first coordinate of a bivariate stochastic fluid process on the event that the second coordinate's level is below a certain threshold. We note that none of these works consider the problem of computing the joint probability of ruin. Some related bivariate problems, in the context of ruin processes with jumps, have been studied in \cite{avram2008exit,avram2008two,badescu2011two,badila2014queues} for the case of proportional reinsurance with common shocks. While the authors are able to compute the probability of ultimate joint ruin for a Cram\'er-Lundberg-type process, they do not analyze the law of their respective ruin times, a more challenging descriptor to analyze.

The primary objective of this research is to explore the impact of the ruin-dependent behavioral switch on the joint behavior of two stochastic fluid processes. We will develop the mathematical framework for the proposed model, analyze its properties, and derive closed-form expressions for key performance measures, specifically, the \emph{joint law} of the ruin times of the coordinates. This descriptor is challenging to compute in a univariate setting (see eg. \cite{asmussen1984approximations,asmussen2002erlangian,stanford2005phase}), and has been recently revisited in a bivariate Cram\'er-Lundberg risk setting through the use of Laguerre series \cite{cheung2023finite}. 

In contrast to previous literature, the dependency is triggered in our case  by ruin rather than being associated with the common shocks or the common underlying processes.
More importantly, our approach to computing descriptors of the bivariate process differs from that in the existing literature in that we consider pathwise approximations and their first return probabilities. We construct a class of approximations that are deemed \emph{compatible}, which, heuristically speaking, maintain the structure of the original path while allowing for a noisy behavioral switching time.

The approximation within the class of compatible pastings has the advantage that it can be paired with a Poissonian observation scheme.  This allows us to  study its associated ruin probabilities via novel matrix-analytic and algorithmic considerations.
In essence, our approach amounts to creating a  scheme for  randomly observing the step  at which the behavioral switching occurs. 
Note that state of the art in uniformization techniques for either time-homogeneous or time-inhomogeneous Markov jump processes does not apply to our bivariate process.  Our  uniformization methodology is thus a key ingredient, which could potentially be extended in higher dimensions.

The organization of this paper is as follows: Section~\ref{sec:model} defines the bivariate ruin-dependent stochastic fluid processes model. In Section~\ref{Sec:pasting}, we devise a family of pathwise approximations for these fluid processes, which we refer to as bivariate compatible pastings. Section~\ref{sec:converge} presents our results regarding the convergence of bivariate compatible pastings, while Section~\ref{sec:main} is devoted to calculating the first return probabilities. We provide the proofs of our main theorems in Section~\ref{sec:proofs}. Finally, Section~\ref{sec:conclusion} offers concluding remarks.
\section{Bivariate model definition}\label{sec:model}

Consider a discrete and finite jump-space $\mathcal{E}$, and a reward function $r:\mathcal{E}\mapsto\mathds{R}\setminus\{0\}$. A univariate \emph{stochastic fluid process} $F=\{F(t)\}_{t\ge 0}$ is  defined as
\begin{equation}\label{eq:defunivF}F(t)=\int_0^t r(J(t))\dd s,\end{equation}
where $J$ is a c\`adl\`ag (right continuous with left limits) jump process on $\mathcal{E}$. For the sake of clarity, here we refer to $F$ as the \emph{level} process and $J$ as the \emph{environmental} process (associated to $F$). If the process $J$ is assumed to be Markovian, say driven by an intensity matrix $\bm{A}$, then the distribution of the  univariate stochastic fluid process is  completely characterized by the triplet $(\mathcal{E},\bm{A}, r)$. Alternatively, we can describe the evolution of the $2$-dimensional process $(F, J)$ by means of a piecewise-deterministic Markov process (PDMP), introduced in \cite{davis1984piecewise}, as follows. In this framework, we consider a state-space $\mathds{R}\times\mathcal{E}$ and let  $(F, J)$ be specified by the local characteristics:

\begin{itemize}
    \item A vector field $\mathcal{X}$ characterizing the flow of $(F,J)$ between jumps that takes the form
    \begin{align*}
    \mathcal{X}f(y,i) & = r(i)\frac{\partial f}{\partial y}(y,i),\quad y\in\mathds{R}, i\in \mathcal{E};
    \end{align*}
    this implies that on a holding time where $J$ equals $i$, the process $F$ moves uniformly in $\mathds{R}$ at a rate $r(i)$. This ensures that $F$ is a continuous process that evolves according to \eqref{eq:defunivF}.
    \item Jump intensity $\lambda(y,i)=c_0$ for all $(y,i)\in \mathds{R}\times\mathcal{E}$, where $c_0$ is some fixed constant greater than or equal to $\sup_{i\in\mathcal{E}}\left| A_{ii} \right|$.
    \item Jump kernels of the form
    \begin{align*}
     &Q((y,i);(\dd z, \dd j))=\left\{ \begin{array}{ccc} \tfrac{1}{c_0}A_{ij} &\mbox{if}& y=z, i\neq j,\\
     1+\tfrac{1}{c_0}A_{ij} &\mbox{if}& y=z, i=j,\\
     0&\mbox{if}&y\neq z.
    \end{array}\right.
    \end{align*} 
    The jump intensity and kernels jointly imply that for $J$, a transition from $i$ to $j$ ($i\neq j$) within the time interval $[t,t+\dd t)$ occurs with probability $(c_0\,\dd t)(\tfrac{1}{c_0}A_{ij})=A_{ij}\,\dd t$. Similarly, in the same interval, the probability of no transition in the environmental process is $(1-c_0\,\dd t) + (c_0\,\dd t)(1+ \tfrac{1}{c_0}A_{ii}) = 1 + A_{ii}\,\dd t$, confirming that $J$ is driven by the intensity matrix $\bm{A}$.
\end{itemize}
We emphasize that this specific construction of $J$ is rooted in the concept of uniformization (see for instance, \cite{van2018uniformization}), which fundamentally involves defining a jump process with switch times that occur on a Poissonian grid.

In the bivariate model of interest,  we consider two univariate stochastic fluid process, $F^{(k)}$ for $k\in\{1,2\}$, each one driven by an environmental process $J^{(k)}$. Here, the processes $J^{(1)}$ and $J^{(2)}$ are not Markovian, but rather dependent on certain hitting times of the level processes $F^{(1)}$ and $F^{(2)}$. To be precise, the level processes and their associated environmental processes are independent up to the first time that one of the fluid processes reaches zero. That is, for $k\in\{1,2\}$, we consider the jump-space $\mathcal{E}^{(k)}$, intensity matrix $\bm{A}^{(k)}_{\mathcal{E}^{(k)}\mathcal{E}^{(k)}}$, and reward function $r^{(k)}:\mathcal{E}^{(k)}\mapsto\mathds{R}\setminus \{0\}$. The processes $F^{(1)}$ and $F^{(2)}$ are driven by the triplets $(\mathcal{E}^{(1)}, \bm{A}^{(1)}_{\mathcal{E}^{(1)}\mathcal{E}^{(1)}}, r^{(1)})$ and $(\mathcal{E}^{(2)}, \bm{A}^{(2)}_{\mathcal{E}^{(2)}\mathcal{E}^{(2)}}, r^{(2)})$, respectively, independently of each other up to the \emph{first ruin time} $\tau^{[1]}$, where
\[\tau^{[1]}:=\inf\{s> 0: F^{(1)}(s) = 0\mbox{ or } F^{(2)}(s)=0\}.
\] 
After this point in time, a \emph{behavioral switch} occurs and the environmental process $J^{(k)}$, $k\in\{1,2\}$, will evolve in a new jump-space $\mathcal{S}^{(k)}$ ($\mathcal{E}^{(k)}\cap\mathcal{S}^{(k)}=\emptyset$), with the triplet associated to $F^{(k)}$ now given by $(\mathcal{S}^{(k)}, \bm{A}^{(k)}_{\mathcal{S}^{(k)}\mathcal{S}^{(k)}}, \rho^{(k)})$. Notably, after $\tau^{[1]}$, the driving intensity matrix becomes $\bm{A}^{(k)}_{\mathcal{S}^{(k)}\mathcal{S}^{(k)}}$ and the reward function changes to $\rho^{(k)}:\mathcal{S}^{(k)}\mapsto\mathds{R}\setminus\{0\}$, with the processes $F^{(1)}$ and $F^{(2)}$ evolving independently. We note that $F^{(k)}$ can be expressed as
\begin{equation}\label{eq:integralbivariate1}F^{(k)}(t)=  \int_0^{\tau^{[1]}\wedge t} r^{(k)}(J^{(k)}(s))\dd s + \int_{\tau^{[1]}\wedge t}^{t} \rho^{(k)}(J^{(k)}(s))\dd s,\quad k\in\{1,2\},\, t\ge 0.\end{equation}
For $a,b \in \mathds{R}$, the notation $a\wedge b:=\max\{a,b\}$ represents the maximum of $a$ and $b$, while $a\vee b:=\min\{a,b\}$ denotes the minimum of $a$ and $b$.
\begin{figure}[h!]

  \centering
  \includegraphics[scale=1.7]{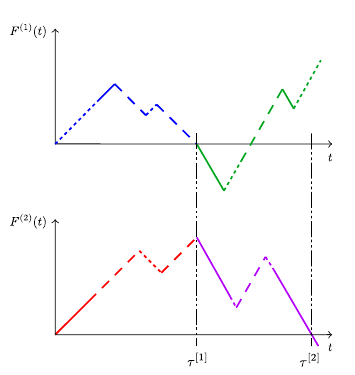}
\caption{Path of the bivariate process $(F^{(1)},F^{(2)})$. The states visited for the environmental process $J^{(k)}$, $k\in\{1,2\}$, are shown via changing line-style (solid, dotted or dashed). The switch from $\mathcal{E}^{(k)}$ to $\mathcal{S}^{(k)}$ at time $\tau^{[1]}$ is visually represented by a change in color: the path for $k=1$ changes from blue to green, while the path for $k=2$ transitions from red to magenta.}\label{fig:bivariate1}
\end{figure}

Stemming from the notion of ruin in risk theory, where a process is deemed ruined if it ever hits (or crosses) level $0$, we term the process $(F^{(1)}, F^{(2)})$ a \emph{ruin-dependent bivariate stochastic fluid process}. Note that the only dependence between the first and second component is triggered by the first passage of either of their levels to  $0$, but besides this, both components essentially evolve independently of each other. Our objective is to study the bivariate law of $(\tau^{[1]},\tau^{[2]})$, where $\tau^{[2]}$ is the time when both stochastic fluid processes have reached zero
\[\tau^{[2]}:=\inf\{s> 0: F^{(1)}(s_1) = 0\mbox{ and } F^{(2)}(s_2)=0\mbox{ for some } s_1,s_2\in[0,s]\}.\]
Figure \ref{fig:bivariate1} provides a visual representation of the aforementioned components.

Without loss of generality, from here on we focus on the case where $J^{(1)}(0)=i^{(1)}$ and $J^{(2)}(0)=i^{(2)}$ with $r^{(1)}(i^{(1)})>0$ and $r^{(2)}(i^{(2)})>0$; other cases can be handled similarly by mirroring the reward rates. In this context, the random times $\tau^{[1]}$ and $\tau^{[2]}$ can be rewritten as exit times:
\begin{align*}
\tau^{[1]}&=\inf\{s> 0: F^{(1)}(s)\notin \mathds{R}_+\mbox{ or } F^{(2)}(s)\notin \mathds{R}_+\}\\
\tau^{[2]}&=\inf\{s> 0: F^{(1)}(s_1) \notin \mathds{R}_+\mbox{ and } F^{(2)}(s_2)\notin \mathds{R}_+\mbox{ for some } s_1,s_2\in[0,s]\},
\end{align*}
where $\mathds{R}_+:=(0,\infty)$.
Similarly to the univariate case, our ruin-dependent bivariate stochastic fluid process can  be represented as a PDMP $X=((F^{(1)}, J^{(1)}),(F^{(2)}, J^{(2)}))$ that is linked to the uniformization method.
The process $X$ now evolves in the state-space 
\[((\mathds{R}_+\times\mathcal{E}^{(1)})\cup (\mathds{R}\times\mathcal{S}^{(1)}))\times ((\mathds{R}_+\times\mathcal{E}^{(2)})\cup (\mathds{R}\times\mathcal{S}^{(2)}));\] each component $(F^{(k)}, J^{(k)})$ evolves in $(\mathds{R}_+\times\mathcal{E}^{(k)}))$ up to $\tau^{[1]}$, and in $(\mathds{R}\times\mathcal{S}^{(k)}))$  afterwards. Unlike the PDMP considered for the univariate case, the state space for $(F^{(k)}, J^{(k)})$ here has a boundary set, $\{0\}\times\mathcal{E}^{(k)}$, that is to be used to trigger the behavioral switch intended for the model. The local characteristics associated with the PDMP $X$ are:
\begin{itemize}
    \item A vector field $\mathcal{X}$ that takes the form
    \begin{align*}
    \mathcal{X}f(y_1,i_1,y_2,i_2) & = (r^{(1)}(i_1)\mathds{1}_{i_1\in\mathcal{E}^{(1)}}+ \rho^{(1)}(i_1)\mathds{1}_{i_1\in\mathcal{S}^{(1)}})\frac{\partial f}{\partial y_1}(y_1,i_1,y_2,i_2)\\
    &\quad + (r^{(2)}(i_2)\mathds{1}_{i_2\in\mathcal{E}^{(2)}}+ \rho^{(2)}(i_2)\mathds{1}_{i_2\in\mathcal{S}^{(2)}})\frac{\partial f}{\partial y_2}(y_1,i_1,y_2,i_2).
    \end{align*}
    Essentially, between jumps and when $(F^{(k)}, J^{(k)})$ is in $\mathds{R}_+\times\{i\}$ for some $i\in\mathcal{E}^{(k)}$, then $F^{(k)}$ uniformly moves within $\mathds{R}_+$ at a rate $r^{(k)}(i)$. Similarly, between jumps and when $(F^{(k)}, J^{(k)})$ is in $\mathds{R}\times\{i\}$ for some $i\in\mathcal{S}^{(k)}$, then $F^{(k)}$ uniformly moves within $\mathds{R}$ at a rate $\rho^{(k)}(i)$. This is consistent with \eqref{eq:integralbivariate1}.
    \item Jump intensity $\lambda:=2\gamma_0$  for some 
    \begin{equation}\label{eq:gamma0def}\gamma_0\ge \sup_{i_1\in\mathcal{E}^{(1)}\cup\mathcal{S}^{(1)}}\left| A^{(1)}_{i_1i_1} \right|\vee \sup_{i_2\in\mathcal{E}^{(2)}\cup\mathcal{S}^{(2)}}\left| A^{(2)}_{i_2i_2} \right|.\end{equation}
    The constant intensity $2\gamma_0$ represents the superposition of two independent Poisson processes of intensity $\gamma_0$, over which the jumps associated to $J^{(1)}$ and $J^{(2)}$ (other than the one triggered by hitting the boundary set) can occur by means of uniformization.
    \item For the interior points $(y_1,i_1)\notin \{0\}\times\mathcal{E}^{(1)}$ and $(y_2,i_2)\notin \{0\}\times\mathcal{E}^{(2)}$, jump kernels of the form
    \begin{align*}
     &Q((y_1,i_1),(y_2,i_2);(\dd z_1, \dd j_1), (\dd z_2, \dd j_2))\\
     &=\left\{ \begin{array}{ccc} \tfrac{1}{2\gamma_0} {A}^{(1)}_{i_1,j_1} &\mbox{if}& y_1=z_1, y_2 = z_2, i_1\neq j_1, i_2=j_2,\\
    \tfrac{1}{2\gamma_0}{A}^{(2)}_{i_2,j_2} &\mbox{if}& y_1=z_1, y_2 = z_2, i_1=j_1, i_2\neq j_2,\\
 1+\tfrac{1}{2\gamma_0}{A}^{(1)}_{i_1,j_1} + \tfrac{1}{2\gamma_0}{A}^{(2)}_{i_2,j_2} &\mbox{if}& y_1=z_1, y_2 = z_2, i_1=j_1, i_2= j_2,\\
 0&\mbox{if} & y_1\neq z_1\mbox{ or }y_2\neq z_2.
    \end{array}\right.
    \end{align*}
    Jointly with the jump intensity, the aforementioned jump kernels imply that a jump of $J^{(k)}$ from $i_k\in\mathcal{E}^{(k)}$ to $j_k\in\mathcal{E}^{(k)}$   within the time interval $[t, t + \dd t)$ (while the other coordinate does not jump) occurs with probability $(2\gamma_0\, \dd t)(\tfrac{1}{2\gamma_0} {A}^{(1)}_{i_k,j_k})={A}^{(1)}_{i_k,j_k}\,\dd t$. The probability that neither process $J^{(1)}$ and $J^{(2)}$ jumps in the same interval is given by 
    \[(1-2\gamma_0\,\dd t) + (2\gamma_0\,\dd t) (1+\tfrac{1}{2\gamma_0}{A}^{(1)}_{i_1,j_1} + \tfrac{1}{2\gamma_0}{A}^{(2)}_{i_2,j_2})=1 - ({A}^{(1)}_{i_1,j_1} + {A}^{(2)}_{i_2,j_2})\,\dd t.\]
    This confirms that up to $\tau^{[1]}$, the processes $J^{(1)}$ and $J^{(2)}$ evolve in an independent fashion according to the intensity matrices $\bm{A}^{(1)}_{\mathcal{E}^{(1)}\mathcal{E}^{(1)}}$ and $\bm{A}^{(2)}_{\mathcal{E}^{(2)}\mathcal{E}^{(2)}}$, respectively. Similar arguments confirm as well that after $\tau^{[1]}$, $J^{(1)}$ and $J^{(2)}$ evolve independently driven by $\bm{A}^{(1)}_{\mathcal{S}^{(1)}\mathcal{S}^{(1)}}$ and $\bm{A}^{(2)}_{\mathcal{S}^{(2)}\mathcal{S}^{(2)}}$. 
    Finally, for the boundary points $(y_1,i_1)\in \{0\}\times\mathcal{E}^{(1)}$ and $(y_2,i_2)\in \{0\}\times\mathcal{E}^{(2)}$,
    \begin{align*}
     &Q((y_1,i_1),(y_2,i_2);(\dd z_1, \dd j_1), (\dd z_2, \dd j_2))=\left\{ \begin{array}{ccc} P^{(1)}_{i_1,j_1}P^{(2)}_{i_2,j_2} &\mbox{if}&  y_2 = z_2, y_2 = z_2,\\
 0&\mbox{if} & y_1\neq z_1\mbox{ or }y_2\neq z_2,
    \end{array}\right.
    \end{align*}
for a transition matrix $\bm{P}^{(k)}=\{P^{(k)}_{ij}\}_{i\in\mathcal{E}^{(k)}, j\in\mathcal{S}^{(k)}}$, which establishes how the process $J^{(k)}$ switches its state at $\tau^{[1]}$.
\end{itemize}

As mentioned earlier, our main objective is to compute the distribution of $(\tau^{[1]},\tau^{[2]})$, a task that presents several challenges. 
\begin{itemize}
\item To calculate the univariate law of $\tau^{[1]}$, we need to separately define the law of the hitting time of $0$ for $F^{(1)}$ and $F^{(2)}$, considering their independent nature up to $\tau^{[1]}$. Existing literature primarily provides such an analysis via their Laplace transforms, which necessitates inversion, a step that demands substantial computational resources. For a comprehensive analysis of inversion methods applicable to the first return probabilities of stochastic fluid processes, see \cite{telek2022transient}. 

\item Once the univariate distribution of $\tau^{[1]}$ is computed, in principle, we could use the law of total probability and condition on the paths up to $\tau^{[1]}$ in order to compute the joint distribution of $\tau^{[1]}$ and $\tau^{[2]}$. The resulting conditional is equivalent to analyzing the hitting time of $0$ for a univariate stochastic fluid process whose level starts from either $F^{(1)}(\tau^{[1]})$ or $F^{(2)}(\tau^{[1]})$ (whichever is non-zero). Although the existing literature provides ways to compute first passage probabilities with a fixed initial level (see e.g., \cite{asmussen1995stationary}), no clear methodology exists for computing the distribution of this random level.
\end{itemize}
In this study, we follow an alternative approach that involves constructing an approximation scheme for $\tau^{[1]}$ and $\tau^{[2]}$. Specifically, we create a family of pathwise approximations for $F^{(1)}$ and $F^{(2)}$, which we refer to as bivariate compatible pastings. These approximations allow for a behavioral switching time that is distinct, albeit close to the original. Remarkably, allowing for this behavior provides the flexibility needed to study the first return times of the approximated paths through Poissonian and algorithmic considerations. Ultimately, this approach will lead to an approximation of the law of $(\tau^{[1]},\tau^{[2]})$.

\section{Construction of bivariate compatible pasting}\label{Sec:pasting}
Consider a modification of the paths of $F^{(1)}$ and $F^{(2)}$, denoted as  $\widetilde{F}^{(1)}$ and $\widetilde{F}^{(2)}$, that allows for the behavioral switch to occur at epochs in time which are \emph{close to} (but not necessarily exactly at) $\tau^{[1]}$. These processes $\widetilde{F}^{(1)}$ and $\widetilde{F}^{(2)}$ can be thought to have a triggering time which is noisy and  may not occur at the exact moment of hitting time $0$ for either component, but otherwise, they behave similarly to the original processes  $F^{(1)}$ and $F^{(2)}$. We aim to make this idea precise, as well as to measure the uniform distance between the original process ${F}^{(1)}$ and the approximation $\widetilde{F}^{(1)}$.

Our first task is to define a class of pathwise modifications, which we term \emph{compatible pastings}. These are fully characterized by the paths of the PDMP $X=((F^{(1)}, J^{(1)}),(F^{(2)}, J^{(2)}))$ (appended with the Poissonian times at which their uniformization takes place), as well as new random times $\sigma^{(1)}$ and $\sigma^{(2)}$, which are to be used as the behavioral switching times for the first and second fluid coordinates, respectively. Essentially, for a compatible pasting, we stitch certain intervals of the original path onto different time and space points in such a way that the new paths remain consistent and close to the original ones. We will delve into these details next.

Let $\Theta^{(1)}=\{\thz^{(1)}_\ell\}_{\ell\ge 0}$ and $\Theta^{(2)}=\{\thz^{(2)}_\ell\}_{\ell\ge 0}$ be the Poisson times over which the intensity-induced jumps of $J^{(1)}$ and $J^{(2)}$ occur, respectively, in the sense of the PDMP or uniformization construction. Recall that each correspond to a Poisson process of intensity $\gamma_0$ of the form \eqref{eq:gamma0def}, and since they are independent of each other, their superposition yields a Poisson process of intensity $2\gamma_0$. 

From here on, the superscript $[1]$' denotes the index of the first process $F^{(k)}$ that defaults, and $[2]$' denotes the index of the second process $F^{(k)}$ that defaults. In this sense, the superscript $[1]$' is interchangeable with $(k)$' if and only if $F^{(k)}(\tau^{[1]})=0$. We apply this notation to all associated processes and random variables, e.g., $F^{[k]}, J^{[k]}, \widetilde{F}^{[k]}, \widetilde{J}^{[k]}, \Theta^{[k]}$, etc. With this notation in place, studying $\tau^{[k]}$ is the same as analyzing the first return times of $F^{[k]}$.

For $k\in\{1,2\}$, let us  define $\ell^{(k)}$ by
\[\ell^{(k)}=\arg_\ell\left\{\tau^{[1]} \in \big[\thz_{\ell}^{(k)}, \thz_{\ell+1}^{(k)}\big)\right\},\]
with the convention that $\ell^{(k)}=\infty$ if $\tau^{[1]}=\infty$. In other words, the time of the first default happens between the $\ell^{(k)}$-th and $(\ell^{(k)} + 1)$-th arrivals of $\Thz^{(k)}$. We then consider the following (random) arbitrary times  $\sigma^{[1]}$ and $\sigma^{[2]}$, which satisfy 
\begin{equation}\label{eq:compatible}
 \sigma^{[1]}\in \big[\tau^{[1]}, \thz_{\ell^{[1]}+1}^{[1]}\big) \ \text{and} \ \sigma^{[2]}\in \big[\thz_{\ell^{[2]}}^{[2]}, \thz_{\ell^{[2]}+1}^{[2]}\big).
\end{equation}
Respectively, the random times $\sigma^{[1]}$ and $\sigma^{[2]}$ are intended to be used as the new behavioral switching times for the approximating processes $\widetilde{F}^{[1]}$ and $\widetilde{F}^{[2]}$ associated to ${F}^{[1]}$ and ${F}^{[2]}$. Note that under \ref{eq:compatible}, the new behavioral switching time $\sigma^{[1]}$ for $\widetilde{F}^{[1]}$ can only occur after $\tau^{[1]}$, but before the subsequent Poisson epoch in $\Thz^{[1]}$. This ensures that the modified level process $\widetilde{F}^{[1]}$ crosses level $0$, while also stipulating it to be triggered before the next Poissonian time at which a jump could occur. Equation \ref{eq:compatible} also ensures that the new switching time $\sigma^{[2]}$ applicable to $F^{[2]}$ happens within the same Poissonian interval in $\Thz^{[2]}$ in which $\tau^{[1]}$ is located.

\begin{definition}[Bivariate compatible pasting] For random times $\sigma^{[1]}$ and $\sigma^{[2]}$ that satisfy \eqref{eq:compatible},
we define the \emph{bivariate compatible pasting} $(\widetilde{F}^{[1]},\widetilde{F}^{[2]})$ of $({F}^{[1]},{F}^{[2]})$ as the following path transformation:
\begin{itemize}
    \item For all $t\in [0,\tau^{[1]}\wedge \sigma^{[k]})$,  \[\widetilde{J}^{[k]}(t) = J^{[k]}(t) \ \ \text{and} \ \ \widetilde{F}^{[k]}(t) = F^{[k]}(t).\]
    \item For all $t\in [\tau^{[1]} \wedge \sigma^{[k]},\sigma^{[k]})$, 
    \[\widetilde{J}^{[k]}(t) = J^{[k]}(\tau^{[1]}\wedge \sigma^{[k]}-) \ \ \text{and} \ \ \widetilde{F}^{[k]}(t) = \widetilde{F}^{[k]}(\tau^{[1]}\wedge \sigma^{[k]}-) + \int_{\tau^{[1]}\wedge \sigma^{[k]}}^t r^{[k]}(\widetilde{J}^{[k]}(s))\dd s.\] 
    \item For all $t\in [\sigma^{[k]}, \thz_{\ell_{k}+1}^{[k]})$, \[\widetilde{J}^{[k]}(t) = J^{[k]}(\tau^{[1]}) \ \ \text{and} \ \  \widetilde{F}^{[k]}(t) = \widetilde{F}^{[k]}(\sigma^{[k]} -) + \int_{\sigma^{[k]}}^t \rho^{[k]}(\widetilde{J}^{[k]}(s))\dd s.\] 
    \item For all $t\in [\thz_{\ell_{k}+1}^{[k]},\infty)$, \[\widetilde{J}^{[k]}(t) = J^{[k]}(t) \ \ \text{and} \ \  \widetilde{F}^{[k]}(t) = \widetilde{F}^{[k]}(\thz_{\ell_{k}+1}^{[k]}-) + \int_{\thz_{\ell_{k}+1}^{[k]}}^t \rho^{[k]}(\widetilde{J}^{[k]}(s))\dd s.\]
\end{itemize}    
\end{definition}

In essence, $\widetilde{J}^{[k]}$ coincides with $J^{[k]}$ on $[0, \thz^{[k]}_{\ell_k})\cup[\thz^{[k]}_{\ell_k+1},\infty)$, and on $[\thz^{[k]}_{\ell_k}, \thz^{[k]}_{\ell_k+1})$, $J^{[k]}$ and $\widetilde{J}^{[k]}$ visit the same two states, with the switch happening at $\sigma^k$ instead of $\tau^{[1]}$; see Figure \ref{fig:pasting1}. Meanwhile, $\widetilde{F}^{[k]}$ is defined in such a way that
\[\widetilde{F}^{[k]}(t)= \int_0^{\sigma^{[k]} \wedge t} r^{[k]}(\widetilde{J}^{[k]}(s))\dd s + \int_{\sigma^{[k]}\wedge t}^{t} \rho^{[k]}(\widetilde{J}^{[k]}(s))\dd s.\] 
\begin{figure}[h!]
  \centering
  \includegraphics[scale=1.5]{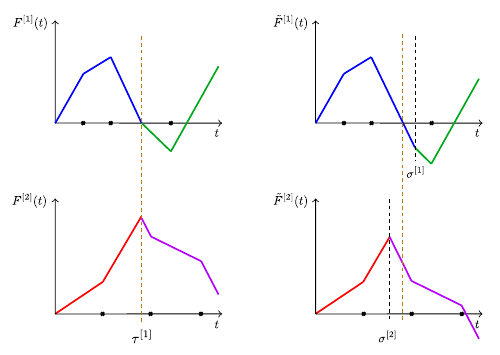}
\caption{Displayed in the left column is a path of a ruin-dependent bivariate stochastic fluid process $(F^{[1]}, F^{[2]})$. In the right column, we present its associated compatible pasting $(\widetilde{F}^{[1]}, \widetilde{F}^{[2]})$, with behavioral switches occurring at times $\sigma^{[1]}$ and $\sigma^{[2]}$ respectively.  Note that besides the differences in time of these behavioral switches, the environmental processes $\widetilde{J}^{[1]}$ and $\widetilde{J}^{[2]}$ jump identically to ${J}^{[1]}$ and ${J}^{[2]}$, respectively.}\label{fig:pasting1}
\end{figure}

Let us stress that, if \ref{eq:compatible} were not to hold, then one cannot consistently build a compatible pasting. For instance, if $\sigma^{[k]} < \thz^{[k]}_{\ell_k}$, then at $\thz^{[k]}_{\ell_k}$ the process $\widetilde{J}^{[k]}$ would need to switch from $\mathcal{E}^{[k]}$ to $\mathcal{S}^{[k]}$. The lack of such a switch for $J^{[k]}$, which at $\thz^{[k]}_{\ell_k}$ remains within  $\mathcal{E}^{[k]}$, makes it impossible to choose a switching for $\widetilde{J}^{[k]}$ that retains the same probabilistic features as $J^{[k]}$. Similar considerations are present when $\sigma^{[k]} > \thz^{[k]}_{\ell_k+1}$; see Figure \ref{fig:counter1} for an explicit example were the compatible pasting fails to hold. 
\begin{figure}[h!]
  \centering
  \includegraphics[scale=2]{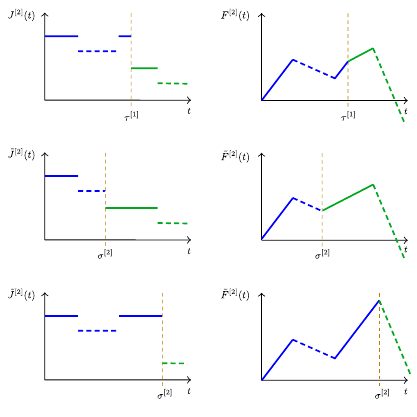}
\caption{On the left is the environmental process $J^{[k]}$ with a behavioral switch at $\tau^{[1]}\in [\thz^{[k]}_{\ell_k}, \thz^{[k]}_{\ell_k+1})$ (upper figures), and ``failed'' new behavioral switches at $\sigma^{(k)'}< \thz^{[k]}_{\ell_k}$ (upper figures) and respectively  $\sigma^{(k)''}> \thz^{[k]}_{\ell_k+1}$ (lower figures). On the right are their corresponding level processes. Note that for both ``failed'' new behavioral switches,  transitions of $J^{[k]}$ within $\mathcal{E}^{[k]}$ were "erased" and the associated processes $\widetilde{F}^{[k]}$ exhibit  peaks that do not resemble the shape of the original process $F^{[k]}$.}\label{fig:counter1}
\end{figure}

The main advantage of compatible pastings is that finding bounds for the (uniform) distance between the original path $F^{[k]}$ and $\widetilde{F}^{[k]}$ is straightforward. Partitioning into cases, 

\begin{align*}
     &|F^{[k]}(t)-\widetilde{F}^{[k]}(t)|\\
     &=\left\{ \begin{array}{ccc} 0 &\mbox{for}& t\in [0,\tau^{[1]}\wedge \sigma^{[k]}),\\
   (t-\tau^{[1]}\wedge \sigma^{[k]})\,\left|r^{[k]}(\widetilde{J}^{[k]}(t)) - \rho^{[k]}(\widetilde{J}^{[k]}(t))\right|  &\mbox{for}& t\in [\tau^{[1]}\wedge \sigma^{[k]},\tau^{[1]}\vee \sigma^{[k]}),\\
\left|F^{[k]}(\tau^{[1]}\vee \sigma^{[k]}-)-\widetilde{F}^{[k]}(\tau^{[1]}\vee \sigma^{[k]}-)\right| &\mbox{for}& t\in [\tau^{[1]}\vee \sigma^{[k]},\infty).
    \end{array}\right.
    \end{align*}

This leads us to the following important result that establishes $\widetilde{F}^{[k]}$ as a uniform approximation of $F^{[k]}$ in the case $\sigma^{[k]}$ is close to $\tau^{[1]}$.
\begin{theorem}\label{th:compatibledistance}
Assume that \ref{eq:compatible} holds. Then we have
\begin{equation}\label{eq:distancecomp1}\sup_{t\ge 0}\left|F^{[k]}(t)-\widetilde{F}^{[k]}(t)\right|\le |\sigma^{[k]}-\tau^{[1]}|\left(\max_{i\in\mathcal{E}^{[k]}}\left\{r^{[k]}(i)\right\} + \max_{i\in\mathcal{S}^{[k]}}\left\{\rho^{[k]}(i)\right\}\right).\end{equation}
\end{theorem}

Remarkably, Theorem \ref{th:compatibledistance} implies that the uniform distance between $F^{[k]}$ and $\widetilde{F}^{[k]}$ is proportional to $|\sigma^{[k]}-\tau^{[1]}|$. As a result, the convergence analysis between the \emph{paths} $F^{[k]}$ and $\widetilde{F}^{[k]}$ can be replaced with a study of the convergence of the \emph{random variables} $\sigma^{[k]}$ to $\tau^{[1]}$. In the following section, we quantify this distance when $\sigma^{[1]}$ and $\sigma^{[2]}$ are considered as the arrival times of a high-frequency Poisson process. This particular choice enables us to develop a tractable bivariate model where the first passage probabilities can be explicitly calculated. This will be discussed further in Section \ref{sec:main}.

\section{Convergence of bivariate compatible pastings}\label{sec:converge}

Let us consider two independent Poisson processes $\Thh^{(1)}$ and $\Thh^{(2)}$ on $\mathds{R}_+$, each  assumed to have high intensity $\gamma-\gamma_0\gg 0$. For $k\in\{1,2\}$, we let $\Thg^{(k)}=\{\thg^{(k)}_\ell\}_{\ell\ge 0}$ be the superposition of the Poisson processes $\Thz^{(k)}$ and $\Thh^{(k)}$, resulting in a Poisson process of intensity $\gamma$. Roughly speaking, our first aim is to define the behavioral switching time $\sigma^{[1]}$ of  $\widetilde{F}^{[1]}$ to be the Poissonian arrival in $\widetilde{\Theta}^{[1]}$ that occurs straight after $\tau^{[1]}$, say the $\ell_*$-th. Secondly, we let $\sigma^{[1]}$ be the $\ell_*$-th observation in $\widetilde{\Theta}^{[2]}$; see Figure \ref{fig:CompatiblePoisson}. 
\begin{figure}[h!]
  \centering
  \includegraphics[scale=2]{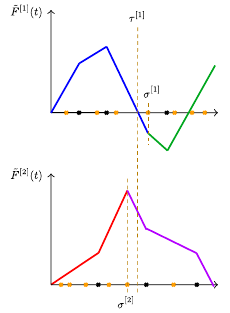}
\caption{Paths of compatible pastings with behavioral switching times occurring a high frequency Poisson grid. For $k\in\{1,2\}$, the Poissonian arrivals of $\Thz^{[k]}$ are shown with black crosses, while those of $\Thh^{[k]}$ are shown in orange; those of $\Thg^{[k]}$ are the superposition of all the crosses. Here $\ell_*=6$, so that $\sigma^{[1]}=\thh^{[1]}_6$ and $\sigma^{[1]}=\thh^{[1]}_6$.}\label{fig:CompatiblePoisson}
\end{figure}
Heuristically, if $\gamma$ is a high-scale parameter leading to a high-frequency Poisson grid $\Thg^{[1]}$, we can anticipate $\sigma^{[1]}$ to occur shortly after $\tau^{[1]}$. Moreover, since the grids  $\Thg^{[1]}$ and $\Thg^{[2]}$ have the exact same distributional properties, we can also expect their $\ell_*$-th arrivals, corresponding to $\sigma^{[1]}$ and $\sigma^{[2]}$, to be closely aligned.

Let us now provide precise definitions in order to rigorously analyze the convergence of the aforementioned high-frequency Poissonian observations scheme for $\sigma^{[1]}$ and $\sigma^{[2]}$, for which we will need some auxiliary random variables first. On $\tau^{[1]}<\infty$, define
\begin{align*}
    \ell_*=\inf\{\ell\ge 1 : F^{[1]}(\thg^{[1]}_\ell)<0\},\qquad\sigma^{[1]}_*=\thg^{[1]}_{\ell_*},\qquad\sigma^{[2]}_*=\thg^{[2]}_{\ell_*};
\end{align*}
note that $\ell_*$ can be alternatively characterized by $\ell_*=\inf\{\ell\ge 1 : \tau^{[1]}<\thg^{[1]}_\ell\}$.
In the following theorem, we provide a strong rate of convergence of the random variables $\sigma^{[1]}_*$ and $\sigma^{[2]}_*$ to $\tau^{[1]}$ as $\gamma\rightarrow\infty$.

\begin{theorem}\label{th:convergence}
Fix some $\epsilon\in (0,1)$ and $q>0$. Then, there exist a function $\delta(\gamma, \epsilon, q)$ that is asymptotically proportional to $(\log \gamma)\gamma^{-1/2+ \epsilon/2}$, and function $K(\gamma, \epsilon, q)$ that is asymptotically proportional to $\gamma^\epsilon$, such that for $k=1,2$,
\begin{equation}\mathds{P}\left(|\tau^{[1]} - \sigma^{[k]}_*|> \delta(\gamma, \epsilon, q), \tau^{[1]}\le K(\gamma, \epsilon, q)\right) = o(\gamma^{-q})\quad\mbox{as}\quad\gamma\rightarrow\infty.\label{eq:Poisson1}\end{equation}
\end{theorem}
\begin{proof}
See Section~\ref{proof:th:convergence}.
\end{proof}

As a corollary of Theorem~\ref{th:convergence}, we have the following convergence result for the first ruin time. 

\begin{Corollary}\label{cor:strongtau1}
On $\tau^{[1]}<\infty$,
\begin{equation}\label{eq:strongtau1}
 \sigma^{[k]}_* \rightarrow \tau^{[1]}\quad\mbox{as}\quad \gamma\rightarrow\infty\quad\mbox{for}\quad k\in\{1,2\},
 \end{equation}
  with the convergence holding in an almost sure sense.
\end{Corollary}
\begin{proof}
Take $q=2$ and $\epsilon\in (0,1)$. Then
\[\sum_{\gamma=1}^\infty \mathds{P}\left(|\tau^{[1]} - \sigma^{[k]}_*|> \delta(\gamma, \epsilon, q), \tau^{[1]}\le K(\gamma, \epsilon, q)\right) < \infty,\]
so that by the Borel-Cantelli lemma, 
\[\mathds{P}\left(|\tau^{[1]} - \sigma^{[k]}_*|> \delta(\gamma, \epsilon, q), \tau^{[1]}\le K(\gamma, \epsilon, q)\mbox{ occurs infinitely often for $\gamma=1,2,\dots$}\right)=0.\]
This in turn implies that 
\[\mathds{P}\left(\lim_{k\rightarrow\infty} \sigma^{[k]}_* \neq \tau^{[1]}, \tau^{[1]}<\infty\right)=0.\]
\end{proof}
Finally, to guarantee that the new behavioral switching times $\sigma^{[1]}$ and $\sigma^{[2]}$ satisfy Equation \eqref{eq:compatible}, we  define
\begin{align*}
    (\sigma^{[1]},\sigma^{[2]})&=\left\{\begin{array}{ccc} (\sigma^{[1]}_*,\sigma^{[2]}_*) & \mbox{if}& \sigma^{[1]}_*\in [\thz_{\ell^{(1)}}^{[1]}, \thz_{\ell^{(1)}+1}^{[1]}) \mbox{ and }\sigma^{[2]}_*\in [\thz_{\ell^{(2)}}^{[2]}, \thz_{\ell^{(2)}+1}^{[2]}),\\
    (\tau^{[1]},\tau^{[1]}) &&\mbox{otherwise}.
    \end{array}\right.
\end{align*}
Note that \eqref{eq:strongtau1} also holds if $\sigma^{[k]}_*$ is replaced by $\sigma^{[k]}$. Thus, employing \eqref{eq:distancecomp1}, we obtain that $\widetilde{F}^{[k]}$ converges uniformly (over $[0,\infty)$) to $F^{[k]}$.

Now that we can guarantee that the proxy $\widetilde{\tau}^{[1]}:=\sigma^{[1]}$ converges to the first ruin time $\tau^{[1]}$ in an almost sure sense as $\gamma\rightarrow\infty$, we define the approximation $\widetilde{\tau}^{[2]}$ to $\tau^{[2]}$ in a similar manner by taking
\begin{align*}
n_* = \inf\{n\ge 1 : \tilde{F}^{[2]}(\thg^{[2]}_n)<0\},\qquad \widetilde{\tau}^{[2]}=\thg^{[2]}_{n_*}.
\end{align*}
Below we present the (bivariate) convergence result which will be the basis of our forthcoming analysis.
\begin{theorem}\label{th:convergence2}
On $\tau^{[1]}, \tau^{[2]}<\infty$, the bivariate vectors $(\widetilde{\tau}^{[1]}, \widetilde{\tau}^{[2]})$ and $(\ell_*/\gamma, n_*/\gamma)$ converge in probability to $(\tau^{[1]}, \tau^{[2]})$ as $\gamma\rightarrow\infty$.
\end{theorem}
\begin{proof}
See Section~\ref{proof:th:convergence2}.
\end{proof}

We finalize this section by pointing out that, since $\tau^{[1]}$ is almost surely an interior point of $[\thz_{\ell^{(1)}}^{[1]}, \thz_{\ell^{[1]}+1}^{[1]}) \cap [\thz_{\ell^{[2]}}^{[2]}, \thz_{\ell^{(2)}+1}^{[2]})$, then \eqref{eq:strongtau1} implies that the probability of the set 
\begin{equation}\label{eq:probcomp1}\mathds{P}\left((\sigma^{[1]},\sigma^{[2]})\neq(\sigma^{[1]}_*,\sigma^{[2]}_*)\right)\rightarrow 0\quad\mbox{as}\quad \gamma\rightarrow\infty.\end{equation}
In other words, the Poissonian mechanism to track the new behavioral switches laid out by $\sigma^{[1]}_*$ and $\sigma^{[2]}_*$ satisfies Equation \eqref{eq:compatible} everywhere but in the asymptotically null set $\{(\sigma^{[1]},\sigma^{[2]})\neq(\sigma^{[1]}_*,\sigma^{[2]}_*)\}$. Moreover, ouside of this set, the process $(\widetilde{F}^{[1]},\widetilde{F}^{[2]})$ can be easily explained in terms of a uniformization scheme as follows. 

For $k=1,2$, underlying $\widetilde{F}^{[k]}$, there is a Poisson process $\Thg^{[k]}$ of parameter $\gamma$, which itself is the superposition of two independent Poisson processes, $\Thz^{[k]}$ and $\Thh^{[k]}$, of intensities $\gamma_0$ and $\gamma-\gamma_0$, respectively. At each  arrival point in $\Thg^{[k]}$, one of two things can occur: either it belongs to $\Theta_0$ and the process $\widetilde{J}^{[k]}$ switches state at such point according to the transition matrix $\bm{I}+\tfrac{1}{\gamma_0}\bm{A}^{[k]}$, or it belongs to $\Thz^{[k]}$ and it does not change state unless this arrival is exactly equal to $\sigma^{[k]}$, point at which $\widetilde{J}^{[k]}$ will switch states according to $\bm{P}^{[k]}$. Since the probability of an arrival belonging to $\Thz^{[k]}$ is $\gamma_0/\gamma$, and it belonging to $\Thh^{[k]}$ is $(\gamma-\gamma_0)/\gamma$, then we can regard all the arrivals in $\Thg$ (that are different from $\sigma^{[k]}$) as uniformization points where a transitions occur according to the matrix
\[\frac{\gamma_0}{\gamma}\left(\bm{I}+\frac{1}{\gamma_0}\bm{A}^{[k]}\right) + \frac{\gamma-\gamma_0}{\gamma}\bm{I}=\bm{I}+\frac{1}{\gamma}\bm{A}^{[k]}.\]
Moreover, the dependence between $\widetilde{F}^{[1]}$ and $\widetilde{F}^{[2]}$ exists only through $\ell_*$, the arrival number in $\Thg^{[1]}$ that triggers the behavioral switch for both components. Taking Theorem \ref{proof:th:convergence2} into account, we can then study the bivariate vector $(\tau^{[1]},\tau^{[2]})$ through $(\ell_*, n_*)$. In fact, for $x,y\ge 0$ and sufficiently large $\gamma$,
\begin{align}
&\mathds{P}(\tau^{[1]}\le x, \tau^{[2]}\le y, F^{[1]}= F^{(1)}\mbox{ and }F^{[2]}=F^{(2)})\nonumber\\
& \quad\approx \mathds{P}(\ell_*/\gamma\le x, n_*/\gamma\le y, \widetilde{F}^{[1]}= \widetilde{F}^{(1)}\mbox{ and }\widetilde{F}^{[2]}=\widetilde{F}^{(2)})\nonumber\\
& \quad = \sum_{\ell=1}^{\lfloor \gamma x\rfloor} \sum_{n=\ell}^{\lfloor \gamma y\rfloor} \mathds{P}(\ell_*=\ell, n_*=n,  \widetilde{F}^{[1]}= \widetilde{F}^{(1)}\mbox{ and }\widetilde{F}^{[2]}=\widetilde{F}^{(2)})\nonumber\\
& \quad = \sum_{\ell=1}^{\lfloor \gamma x\rfloor} \sum_{n=\ell}^{\lfloor \gamma y\rfloor} \mathds{P}(\mbox{Ruin of $\widetilde{F}^{(1)}$ is confirmed at $\thg^{(1)}_\ell$})\nonumber\\
&\qquad\qquad\qquad\times \mathds{P}( \mbox{Ruin of $\widetilde{F}^{(2)}$ is confirmed at $\thg^{(2)}_n$}\mid \mbox{behavioral switch of $\widetilde{F}^{(2)}$ occurs at $\thg^{(2)}_\ell$}).
\label{eq:ondime1}
\end{align}
 Remarkably, the expression (\ref{eq:ondime1}) offers an approximation scheme which rests on computing one-dimensional quantities only. Indeed, for each summand, the first element of the product corresponds to the first return probabilities of $F^{(1)}$ occurring just before the $\ell$-th Poissonian step. Meanwhile, the second term in the product corresponds to the first return probabilities of the process $F^{(2)}$ just before the $n$-th Poissonian step, conditioned on a behavioral switch occurring at the $\ell$-th point. As we will see in the following section, extending the ideas of \cite{bean2019finite,amini2023duration}, both quantities are algorithmically tractable to compute.

\section{First return probabilities}\label{sec:main}

By virtue of the results in the previous section,  the problem of approximating  the law of $(\tau^{[1]},\tau^{[2]})$ reduces to  computing one-dimensional ruin probabilities via \eqref{eq:ondime1} for the compatible pasting approximation of the original process. We can now proceed to develop the matrix-analytic methodology. In this section we relinquish all the superscripts to ease notation and fix some large $\gamma$. Quantity $\widetilde{F}$ will denote a \emph{univariate} stochastic fluid process, which has an associated jump process $\widetilde{J}$  evolving in a state space $\mathcal{E}\cup\mathcal{S}$ ($\mathcal{E}\cap\mathcal{S}=\emptyset$) with $\widetilde{F}$ having instantaneous rewards of the form $r:\mathcal{E}$, when $\widetilde{J}$ is in $\mathcal{E}$, and of the form $\rho:\mathcal{E}$ when $\widetilde{J}$ is in $\mathcal{S}$. Moreover, the process $\widetilde{J}$ evolves by means of uniformization on top of a Poisson process $\Thg=\{\thg_\ell\}_{\ell \ge 0}$ of parameter $\gamma$, with transitions within $\mathcal{E}$ governed by the matrix $\bm{B}_{\mathcal{E}\mathcal{E}}:=\bm{I} + \tfrac{1}{\gamma}\bm{A}_{\mathcal{E}\mathcal{E}}$, transitions within $\mathcal{S}$ governed by the matrix $\bm{B}_{\mathcal{S}\mathcal{S}}:=\bm{I} + \tfrac{1}{\gamma}\bm{A}_{\mathcal{S}\mathcal{S}}$, and transitions from $\mathcal{E}$ to $\mathcal{S}$ (triggered by an external behavioral switch at a Poissonian time in $\Thg$) governed by the matrix $\bm{B}_{\mathcal{E}\mathcal{S}}:=\bm{P}$.


Define the sets
\[\ME^{\pm}=\{i\in\ME: (\pm 1)r(i)>0\},\quad \MS^{\pm}=\{i\in\MS: (\pm 1)\rho(i)>0\}.\]
According to the previous classification, partition the transition matrices $\bm{B}_{\ME\ME}$, $\bm{B}_{\ME\MS}$ and $\bm{B}_{\MS\MS}$ into
\begin{equation}
\bm{B}_{\ME\ME}=\begin{pmatrix}\bm{B}_{\ME^+\ME^+}&\bm{B}_{\ME^+\ME^-}\\ \bm{B}_{\ME^-\ME^+}&\bm{B}_{\ME^-\ME^-}\end{pmatrix},\quad \bm{B}_{\ME\MS}=\begin{pmatrix}\bm{B}_{\ME^+\MS^+}&\bm{B}_{\ME^+\MS^-}\\ \bm{B}_{\ME^-\MS^+}&\bm{B}_{\ME^-\MS^-}\end{pmatrix},\quad\mbox{and}
\quad\bm{B}_{\MS\MS}=\begin{pmatrix}\bm{B}_{\MS^+\MS^+}&\bm{B}_{\MS^+\MS^-}\\ \bm{B}_{\MS^-\MS^+}&\bm{B}_{\MS^-\MS^-}\end{pmatrix}.\end{equation}
We now compute the first passage probability matrix of $\widetilde{F}$. Namely, we find 
\[\bm{\Psi}^{(\ell,n)} = \{\psi_{ij}^{(\ell,n)}\,:\, i\in\mathcal{E}^+\cup\mathcal{S}^+, j\in\mathcal{E}^-\cup\mathcal{S}^-\},\]
where
\[\psi_{ij}^{(\ell,n)}=\mathds{P}( \mbox{Ruin of $\widetilde{F}$ is confirmed at $\thg_n$, $\widetilde{J}(\thg_n-)=j$}\mid \mbox{Behavioral switch of $\widetilde{F}$ occurs at $\thg_\ell$}, \widetilde{J}(0)=i).\]
We stress that we allow for $\ell\in \{n, n+1, n+2,\dots\}$ and $\ell\in \{\dots,-2,-1,0\}$, with both events corresponding to a behavioral switch that occurs outside of $[0,\thg_n)$: in the former the switching time is larger or equal to $\thg_n$, and in the latter it is smaller than $0$. In both cases, the switching mechanism becomes irrelevant, since it occurs outside of the interval in question, and thus, it boils down to those instances where the process remains entirely in $\mathcal{E}$ or in $\mathcal{S}$. Particularly, 
\begin{equation}\label{eq:repeat1}\bm{\Psi}^{(n,n)}=\bm{\Psi}^{(n+1,n)}=\bm{\Psi}^{(n+2,n)}=\cdots\quad\mbox{and}\quad\bm{\Psi}^{(0,n)}=\bm{\Psi}^{(-1,n)}=\bm{\Psi}^{(-2,n)}=\cdots.\end{equation} We also note that all the components in \eqref{eq:ondime1} can be readily recovered from $\psi_{ij}^{(\ell,n)}$ by choosing the relevant parameters. Particularly,
\begin{align*}
\mathds{P}&(\mbox{Ruin of $\widetilde{F}$ is confirmed at $\thg_\ell$}| \widetilde{J}(0)=i)  = \sum_{j'\in \mathcal{E}} \psi_{ij'}^{(\ell,\ell)},\\
\mathds{P}&( \mbox{Ruin of $\widetilde{F}$ is confirmed at $\thg_n$}\mid \mbox{Behavioral switch of $\widetilde{F}$ occurs at $\thg_\ell$}, \widetilde{J}(0)=i)  = \sum_{j'\in \mathcal{S}} \psi_{ij'}^{(\ell,n)}.
\end{align*}
In the following we provide an algorithmic method to recursively compute the matrix $\bm{\Psi}^{(\ell, n)}$.

\begin{theorem}\label{th:main}
Define 
\begin{align*}
\bm{H}_{+-}&=\begin{pmatrix} \left\{\tfrac{1}{r(i) + |r(j)|}: i\in\ME^+, j\in\ME^-\right\} & \left\{\tfrac{1}{r(i) + |\rho(j)|}: i\in\ME^+, j\in\MS^-\right\}\\
\bm{0}& \left\{\tfrac{1}{\rho(i) + |\rho(j)|}: i\in\MS^+, j\in\MS^-\right\} \end{pmatrix},\\
\bm{R}_{-}&=\begin{pmatrix} \mathrm{diag}\left\{|r(j)|: j\in\ME^-\right\} &\bm{0}\\
\bm{0}&\mathrm{diag}\left\{|r(j)|: j\in\MS^-\right\} \end{pmatrix},\\
\bm{R}_{+}&=\begin{pmatrix} \mathrm{diag}\left\{r(j): i\in\ME^+\right\} &\bm{0}\\
\bm{0}&\mathrm{diag}\left\{r(i): i\in\MS^+\right\}
\end{pmatrix}.
\end{align*}
For $n\ge 0$ and $\ell\in\{\dots,-2,-1,0,1,2,\dots\}$,
\begin{align*}
\bm{\Psi}^{(\ell,n)}=\bm{Q}^{(\ell,n)}\bm{R}_-,
\end{align*}
where $\bm{Q}^{(\ell, 2)}  =  \bm{B}_{+-}^{(\ell,2)}\odot \bm{H}_{+-}$ with  \begin{align*}
\bm{B}_{+-}^{(\ell,2)}&=\begin{pmatrix}\bm{B}_{\ME^+\ME^-}\mathds{1}_{\ell>1}&\bm{B}_{\ME^+\MS^-}\mathds{1}_{\ell=1}\\\bm{0}&\bm{B}_{\MS^+\MS^-}\mathds{1}_{\ell<1}\end{pmatrix},
\end{align*}
and, for $n\geq 3$,  $\bm{Q}^{(\ell,n)}$ is computed in a recursive manner w.r.t. $n$ by
\begin{align*} 
\bm{Q}^{(\ell, n)} & = \left(\bm{B}_{++}^{(\ell,n)}\,\bm{Q}^{(\ell-1, n-1)}\,\bm{R}_{-}\right)\odot \bm{H}_{+-}\\
&\quad + \sum_{w=2}^{n-2}\left( \bm{R}_{+}\,\bm{Q}^{(\ell, w)}\,\bm{B}_{-+}^{(\ell,n,w)}\, \bm{Q}^{(\ell-w, n-w)}\,\bm{R}_{-}\right)\odot \bm{H}_{+-}\\
&\quad + \left( \bm{R}_{+}\,\bm{Q}^{(\ell, n-1)}\,\bm{B}_{--}^{(\ell,n)}\right)\odot \bm{H}_{+-},
\end{align*}
where 
\begin{align*}
\bm{B}_{++}^{(\ell,n)}&=\begin{pmatrix}\bm{B}_{\ME^+\ME^+}\mathds{1}_{\ell>1}&\bm{B}_{\ME^+\MS^+}\mathds{1}_{\ell=1}\\ \bm{0}&\bm{B}_{\MS^+\MS^+}\mathds{1}_{\ell<1}\end{pmatrix}, \ \ 
\bm{B}_{--}^{(\ell,n)}=\begin{pmatrix}\bm{B}_{\ME^-\ME^-}\mathds{1}_{\ell>n-1}&\bm{B}_{\ME^-\MS^-}\mathds{1}_{\ell=n-1}\\ \bm{0}&\bm{B}_{\MS^-\MS^-}\mathds{1}_{\ell<n-1}\end{pmatrix}
\end{align*}
and
\begin{align*}
\bm{B}_{-+}^{(\ell,n,w)}&=\begin{pmatrix}\bm{B}_{\ME^-\ME^+}\mathds{1}_{\ell>w}&\bm{B}_{\ME^-\MS^+}\mathds{1}_{\ell=w}\\ \bm{0}&\bm{B}_{\MS^-\MS^+}\mathds{1}_{\ell<w}\end{pmatrix}.
\end{align*}
Here, $\odot$ denotes the Hadamard product of matrices (i.e., entrywise multiplication between matrices).
\end{theorem}

\begin{proof}
See Section~\ref{sec:proof:th:main}.
\end{proof}

Besides proposing novel formulae for the first return probabilities of stochastic fluid processes that present behavioral switching, we highlight that the method proposed in Theorem \ref{th:main} is particularly simple and tractable. The $n$-level of the algorithm is completely specified by the previous levels in $\{2,3,\dots,n-1\}$, all in terms of matrix and Hadamard multiplications, which contrasts with other expensive methods that are common in the literature, such as matrix exponentiation, inversion and Riccati solutions (see eg. \cite{latouche2018analysis} for more details). Furthermore, by \eqref{eq:repeat1}, for the $n$-level there are at most $n+1$ different matrices $\bm{\Psi}^{(\ell,n)}$,  corresponding to the cases $\ell\in\{0,1,2,\dots, n\}$.


\section{Proofs}\label{sec:proofs}
This section includes the proofs of Theorem~\ref{th:convergence}, Theorem~\ref{th:convergence2} and Theorem~\ref{th:main}.
\subsection{Proof of Theorem~\ref{th:convergence}}\label{proof:th:convergence}

Let us borrow the technical result from \cite[Lemma 1]{bladt2022strongly}, where it was used in a different context for time-inhomogeneous Markov jump processes. Such a result states that  there exists some $C(\epsilon,q)>0$ such that
\begin{align}
\mathds{P}\left(G_{(k,\lambda)} \right)=o(\gamma^{-q})\quad\mbox{for}\quad G_{(k,\lambda)} = \left\{\max_{\ell\in\{1,\dots,\lfloor \gamma^{1+\epsilon}\rfloor\}}\left|\widetilde{\theta}^{[k]}_\ell - \ell/\gamma\right| \ge C(\epsilon,q)(\log \gamma)\gamma^{-1/2+\epsilon/2}\right\},\label{eq:Poissonaux1}
\end{align}
\begin{align}
\mathds{P}\left(G_{(*,\lambda)}\right)=o(\gamma^{-q})\quad\mbox{for}\quad G_{(*,\lambda)} = \left\{\max_{\ell\in\{1,\dots,\lfloor \gamma^{1+\epsilon}\rfloor\}}\left|\widetilde{\theta}^{[1]}_\ell - \widetilde{\theta}^{[2]}_\ell\right| \ge 2C(\epsilon,q)(\log \gamma)\gamma^{-1/2+\epsilon/2}\right\}.\label{eq:Poissonaux2}
\end{align}
These equations establish the results on the distance between the Poissonian grid and the fixed grid and respectively between the two independent Poissonian grids.
We now exploit this to compute the distance between the  random times $\tau^{[1]}$ and $\sigma^{[1]}_*$ via triangle inequality considerations. 
This is done via the set inclusion
\begin{align}
& D_{\gamma}\cap E_{\gamma} \subseteq \left(D_{\gamma}\cap E_{\gamma}\cap G_{(1, \lambda)}^c\cap G_{(2,\lambda)}^c \cap G_{(*,\lambda)}^c\right) \cup G_{(1, \lambda)}\cup G_{(2,\lambda)} \cup G_{(*,\lambda)},\label{eq:inclusion1}
\end{align}
where
\begin{align*}
D_\gamma=\{|\tau^{[1]} - \sigma^{[1]}_*|> \delta(\gamma, \epsilon, q)\}, \quad E_\gamma=\{\tau^{[1]}\le K(\gamma, \epsilon, q)\}.
\end{align*}


To be more precise, we will show that for an appropriate choice of $\delta(\gamma, \epsilon, q)$ and $K(\gamma, \epsilon, q)$, the set $D_{\gamma}\cap E_{\gamma}\cap G_{(1, \lambda)}^c\cap G_{(2,\lambda)}^c \cap G_{(*,\lambda)}^c$ is empty, so that \eqref{eq:Poisson1} follows from \eqref{eq:Poissonaux1}, \eqref{eq:Poissonaux2}, \eqref{eq:inclusion1} and the subadditivity property of $\mathds{P}$.

First, note that on $G_{(1, \lambda)}^c\cap G_{(2,\lambda)}^c$, both $\widetilde{\theta}^{[1]}_{\lfloor \gamma^{1+\epsilon}\rfloor}$ and $\widetilde{\theta}^{[2]}_{\lfloor \gamma^{1+\epsilon}\rfloor}$ are strictly larger than $(\lfloor \gamma^{1+\epsilon}\rfloor/\gamma) - C(\epsilon,q)(\log \gamma)\gamma^{-1/2+\epsilon/2}=:K(\gamma,\epsilon,q)$. Thus,
\begin{align}
E_\gamma\cap G_{(1, \lambda)}^c\cap G_{(2,\lambda)}^c\subseteq \{\tau^{[1]}< \theta^{[k]}_{\lfloor \gamma^{1+\epsilon}\rfloor} \mbox{ for }k=1,2\}\cap G_{(1, \lambda)}^c\cap G_{(2,\lambda)}^c.\label{eq:inclusionaux1}
\end{align}
Now, the distance between $\tau^{[1]}$ and $\sigma_*^{[1]}$ is less or equal than the distance between $\widetilde{\theta}^{[1]}_{\ell_*-1}$ and $\widetilde{\theta}^{[1]}_{\ell_*}$. On $\{\tau^{[1]}< \theta^{[k]}_{\lfloor \gamma^{1+\epsilon}\rfloor} \mbox{ for }k=1,2\}$, the random variable $\ell_*$ is guaranteed to be smaller than $\lfloor \gamma^{1+\epsilon}\rfloor$. Thus, $\left|\widetilde{\theta}^{[1]}_{\ell_*} - \widetilde{\theta}^{[1]}_{\ell_*-1}\right|$ is bounded by the greatest distance $\left|\widetilde{\theta}^{[1]}_\ell - \widetilde{\theta}^{[1]}_{\ell-1}\right|$ for $\ell\le \lfloor \gamma^{1+\epsilon}\rfloor$, leading to
\begin{align}
|\tau^{[1]} - \sigma^{[1]}_*| &\le \max_{\ell\in\{1,\dots,\lfloor \gamma^{1+\epsilon}\rfloor\}}\left|\widetilde{\theta}^{[1]}_\ell - \widetilde{\theta}^{[1]}_{\ell-1}\right|\nonumber\\&\le \max_{\ell\in\{1,\dots,\lfloor \gamma^{1+\epsilon}\rfloor\}}\left(\left|\widetilde{\theta}^{[1]}_\ell - \ell/\gamma\right| + \left|\ell/\gamma - (\ell-1)/\gamma\right| + \left|(\ell-1)/\gamma - \widetilde{\theta}^{[1]}_{\ell-1}\right| \right)\nonumber\\
& \le 2 \max_{\ell\in\{1,\dots,\lfloor \gamma^{1+\epsilon}\rfloor\}}\left|\widetilde{\theta}^{[1]}_\ell - \ell/\gamma\right| + 1/\gamma.\label{eq:inequalitykappa}
\end{align}
Following similar steps, it is readily verified that on the set $\{\tau^{[1]}< \theta^{[k]}_{\lfloor \gamma^{1+\epsilon}\rfloor} \mbox{ for }k=1,2\}$, 
\begin{align}
|\tau^{[1]} - \sigma^{[2]}_*|&=|\tau^{[1]} - \sigma^{[1]}_*| + |\sigma^{[1]}_*-\sigma^{[2]}_*|\nonumber\\
&=|\tau^{[1]} - \sigma^{[1]}_*| + |\widetilde{\theta}^{[1]}_{\ell_*}-\widetilde{\theta}^{[2]}_{\ell_*}|\nonumber\\
&\le \left( 2 \max_{\ell\in\{1,\dots,\lfloor \gamma^{1+\epsilon}\rfloor\}}\left|\widetilde{\theta}^{[1]}_\ell - \ell/\gamma\right| + 1/\gamma \right) + \left(\max_{\ell\in\{1,\dots,\lfloor \gamma^{1+\epsilon}\rfloor\}}\left|\widetilde{\theta}^{[1]}_\ell - \widetilde{\theta}^{[2]}_{\ell}\right|\right).\label{eq:inequalitykappac}
\end{align}
Thus, on $E_{\gamma}\cap G_{(1, \lambda)}^c\cap G_{(2,\lambda)}^c \cap G_{(*,\lambda)}^c$, employing \eqref{eq:inclusionaux1}, \eqref{eq:inequalitykappa} and  \eqref{eq:inequalitykappac} we get for $k=1,2$
\begin{align*}
|\tau^{[1]} - \sigma^{[k]}_*|&\le \left( 2 \max_{\ell\in\{1,\dots,\lfloor \gamma^{1+\epsilon}\rfloor\}}\left|\widetilde{\theta}^{[1]}_\ell - \ell/\gamma\right| + 1/\gamma \right) + \left( 2 \max_{\ell\in\{1,\dots,\lfloor \gamma^{1+\epsilon}\rfloor\}}\left|\widetilde{\theta}^{[2]}_\ell - \ell/\gamma\right| + 1/\gamma \right)\\
&\qquad+ \left(\max_{\ell\in\{1,\dots,\lfloor \gamma^{1+\epsilon}\rfloor\}}\left|\widetilde{\theta}^{[1]}_\ell - \widetilde{\theta}^{[2]}_{\ell}\right|\right)\\
& \le 6C(\epsilon,q)(\log \gamma)\gamma^{-1/2+\epsilon/2} + 2/\gamma =: \delta(\gamma,\epsilon,q).
\end{align*}
This implies that $D_\gamma\cap E_{\gamma}\cap G_{(1, \lambda)}^c\cap G_{(2,\lambda)}^c \cap G_{(*,\lambda)}^c$ is an empty set, completing the proof.
\begin{Remark}
Note that the convergence in probability of $\sigma^{[1]}_*$ to $\tau^{[1]}$ can be shown by simpler methods, namely, by writing $\sigma^{[1]}_*$ as $\tau^{[1]}+e_{\gamma}$ for $e_\gamma\sim\mbox{Exp}(\gamma)$ and employing convergence in distribution result (this is in fact the avenue we pursue in the proof of Theorem~\ref{th:convergence2}). However, the convergence of $\sigma_*^{[2]}$ to $\tau^{[1]}$ is more challenging since $\sigma_*^{[2]}$ cannot be decomposed as a sum of a stopping time and an exponential random variable. 
\end{Remark}

\subsection{Proof of Theorem~\ref{th:convergence2}}\label{proof:th:convergence2}
The convergence in probability of $\widetilde{\tau}^{[1]}$ to $\tau^{[1]}$ follows from the strong convergence result in Corollary \ref{cor:strongtau1}. The almost sure convergence of $\sigma^{[2]}$ to $\tau^{[1]}$ implies, together with Theorem \ref{th:compatibledistance}, that the paths of $\widetilde{F}^{[2]}$ converge almost surely to those of $F^{[2]}$ uniformly over $[0,\infty)$. This in turn implies that the first entry time $\tau^{[2]}$ for ${F}^{[2]}$ is the almost sure limit of the first entry time to $(-\infty, 0)$, say $\beta^{[2]}$, for  $\widetilde{F}^{[2]}$. By the memoryless property of the exponential distribution, we know that $\widetilde{\tau}^{[2]}-\beta^{[2]}\sim \mbox{Exp}(\gamma)$. Consequently, $|\widetilde{\tau}^{[2]}-\beta^{[2]}|$ converges in distribution to $0$ as $\gamma\rightarrow \infty$, which can be upgraded to convergence in probability (since the limit point is a constant). In summary,
\begin{align*}
&|\widetilde{\tau}^{[1]}-\tau^{[1]}|\rightarrow 0\quad\mbox{in probability as $\gamma\rightarrow\infty$,}\\
&|\widetilde{\tau}^{[2]}-\tau^{[2]}|\le |\widetilde{\tau}^{[2]}-\beta^{[2]}| + |\beta^{[2]}-{\tau}^{[2]}|\rightarrow 0\quad\mbox{in probability as $\gamma\rightarrow\infty$,}
\end{align*}
yielding the convergence in probability of $(\widetilde{\tau}^{[1]}, \widetilde{\tau}^{[2]})$ to $({\tau}^{[1]}, {\tau}^{[2]})$. Also,
\begin{align*}
\left|\widehat{\tau}^{[1]}-\frac{\ell_*}{\gamma}\right|\mathds{1}_{\ell_* \le \lfloor \gamma^{1+\epsilon}\rfloor}\le \max_{\ell\in\{1,\dots,\lfloor \gamma^{1+\epsilon}\rfloor\}}\left|\widetilde{\theta}^{[1]}_\ell - \ell/\gamma\right|\rightarrow 0\quad\mbox{almost surely as $\gamma\rightarrow\infty$,}\\
\left|\widehat{\tau}^{[2]}-\frac{n_*}{\gamma}\right|\mathds{1}_{n_* \le \lfloor \gamma^{1+\epsilon}\rfloor}\le \max_{\ell\in\{1,\dots,\lfloor \gamma^{1+\epsilon}\rfloor\}}\left|\widetilde{\theta}^{[2]}_\ell - \ell/\gamma\right|\rightarrow 0\quad\mbox{almost surely as $\gamma\rightarrow\infty$,}
\end{align*}
which in turn yields the convergence in probability of $(\ell_*/\gamma,n_*/\gamma$ to $({\tau}^{[1]}, {\tau}^{[2]})$.
\subsection{Proof of Theorem~\ref{th:main}}\label{sec:proof:th:main}

Similar to \cite{amini2023duration}, we consider the concept of $n$-bridges, defined as the collection of paths of $\widetilde{F}$ such that, when restricted to its first $n$ observations, the lowest points of these paths are represented by the  two endpoints. These paths resemble a bridge that stands on its endpoints. Contrasting with \cite{amini2023duration}, we have two distinctions:  the underlying environmental process is neither duration dependent, nor time-inhomogeneous (significantly simplifying the computations), and the paths considered for an $n$-bridge in this context must undergo a behavioral switch at the $\ell$ observation. More specifically, define the set of $n$-bridges (with an $\ell$-behavioral switch)
\[\Omega_{\ell,n}=\left\{\max\left\{\widetilde{F}(\thg_{0}), \widetilde{F}(\thg_n)\right\}<\min\left\{\widetilde{F}(\thg_1),\dots, \widetilde{F}(\thg_{n-1})\right\}\right\}\cap \left\{\mbox{Behavioural switch occurs at $\thg_\ell$}\right\}.\]

The cases under consideration depend on decomposing each $n$-bridge into shorter bridges and accordingly translating the behavioral switch within each of these bridges. Moreover, we need to thoughtfully account for the additional case where the switch occurs in the first or second bridge of the decomposition.
The level density of an $n$-bridge that switches at the $\ell$-th observation is defined as
\begin{align*}
\lambda_{ij}^{(\ell,n)}(s)& :=\frac{\partial L_{ij}^{(n)}(s)}{\partial s},\quad L_{ij}^{(\ell,n)}(s):= \mathds{P}\left(\Omega_{\ell, n},\,J(\thg_{n}-)=j,\, \widetilde{F}(\thg_{n})\le s\mid \widetilde{J}(0)=i\right):
\end{align*}
particularly, note that 
\begin{equation}\label{eq:psiint1}\bm{\psi}_{ij}^{(\ell,n)}=\int_{-\infty}^0 \lambda_{ij}^{(\ell,n)}(s)\dd s.
\end{equation} 
Furthermore, for $n\ge 3$ and $w\in\{1,2,\dots, n-1\}$, define
\[E_{w}=\left\{ \widetilde{F}(\thg_w) < \min_{m\in\{ 1,\dots, n-1\}\setminus\{ w\}} \widetilde{F}(\thg_m)\right\};\]
this set represents the paths for which the decomposition in shorter bridges occurs at $w$-th observation. We define the corresponding densities
\begin{align*}
\Gamma_{ij}^{(\ell,n,w)}(s)& :=\frac{\partial G_{ij}^{(\ell,n,w)}(s)}{\partial s},\quad G_{ij}^{(n,w)}(s):= \mathds{P}\left(\Omega_{\ell, n},\,E_w,\,J(\thg_{n}-)=j,\, \widetilde{F}(\thg_{n})\le s\mid \widetilde{J}(0)=i\right).
\end{align*}
represent level densities on the event that decomposition in shorter bridges occurs at $w$-th observation.
Our aim is to find a recursive algorithm that comprise the matrices
\begin{align*}
\bm{\Lambda}^{(\ell,n)}_{\ME^+\ME^-}(s) = \{\lambda^{(\ell,n)}_{ij}(s)\}_{i\in \ME^+, j\in \ME^-}, \quad \bm{\Gamma}^{(\ell,n,w)}_{\ME^+\ME^-}(s) = \{\Gamma^{(\ell,n,w)}_{ij}(s)\}_{i\in \ME^+, j\in \ME^-},\\
\bm{\Lambda}^{(\ell,n)}_{\ME^+\MS^-}(s) = \{\lambda^{(\ell,n)}_{ij}(s)\}_{i\in \ME^+, j\in \MS^-}, \quad \bm{\Gamma}^{(\ell,n,w)}_{\ME^+\MS^-}(s) = \{\Gamma^{(\ell,n,w)}_{ij}(s)\}_{i\in \ME^+, j\in \MS^-}, \\
\bm{\Lambda}^{(\ell,n)}_{\MS^+\MS^-}(s) = \{\lambda^{(\ell,n)}_{ij}(s)\}_{i\in \MS^+, j\in \MS^-}, \quad \bm{\Gamma}^{(\ell,n,w)}_{\MS^+\MS^-}(s) = \{\Gamma^{(\ell,n,w)}_{ij}(s)\}_{i\in \MS^+, j\in \MS^-}.
\end{align*}
which in turn provides us a way to recursively compute $\bm{\Psi}^{(\ell,n)}$ through \eqref{eq:psiint1}. Moreover, note that for $n\ge 3$,
\begin{align}
\bm{\Lambda}^{(\ell,n)}_{\ME^+\ME^-}(s) & =\sum_{w=1}^{n-1} \bm{\Gamma}^{(\ell,n,w)}_{\ME^+\ME^-}(s),\label{eq:LambdasumGamma1}\\
\bm{\Lambda}^{(\ell,n)}_{\ME^+\MS^-}(s) & =\sum_{w=1}^{n-1}  \bm{\Gamma}^{(\ell,n,w)}_{\ME^+\MS^-}(s),\label{eq:LambdasumGamma2} \\
\bm{\Lambda}^{(\ell,n)}_{\MS^+\MS^-}(s) & =\sum_{w=1}^{n-1}  \bm{\Gamma}^{(\ell,n,w)}_{\MS^+\MS^-}(s),\label{eq:LambdasumGamma3}
\end{align}
so that characterizing $\bm{\Gamma}^{(\ell,n,w)}_{\ME^+\ME^-}(s)$, $\bm{\Gamma}^{(\ell,n,w)}_{\ME^+\MS^-}(s)$ and $\bm{\Gamma}^{(\ell,n,w)}_{\MS^+\MS^-}(s)$ is enough.

We aim to analyze all the possibilities on a case-by-case basis. For the sake of enlisting all the cases, we briefly describe them next.
\begin{enumerate}
  \item \textbf{Case $n=2$.} The $n$-paths only experience one upwards and then one downards movement. We differentiate between different $\ell$-behavioral switching as follows:
  \begin{tasks}(3)
        \task[(a)] \textbf{Case $\ell\ge 2$.}
        \task[(b)] \textbf{Case $\ell\le 0$.}
        \task[(c)] \textbf{Case $\ell = 1$.}
    \end{tasks}
  \item \textbf{Case $n\ge 3$.} For these $n$-paths, besides classificating with regards to the $\ell$-behavioral switching, we also differentiate between different cases for $w$, which is what dictates how the $n$-paths get decomposed into smaller paths. The blueprint looks as follows.
  \begin{enumerate}[leftmargin=3mm]
        \item \textbf{Case $\ell\ge n$:}
        \begin{tasks}(3)
          \task[i.] \textbf{Case $w=1$.}
          \task[ii.] \textbf{Case $w=n-1$.}
          \task[iii.] \textbf{Case $w\in\{1,\dots, n-1\}$.}
        \end{tasks}
        \item \textbf{Case $\ell \le 0$:}
        \begin{tasks}(3)
           \task[i.]  \textbf{Case $w=1$.}
           \task[ii.]  \textbf{Case $w=n-1$.}
           \task[iii.]  \textbf{Case $w\in\{1,\dots, n-1\}$.}
        \end{tasks}
        \item \textbf{Case $\ell\in\{1,\dots, n-1\}$:}
        \begin{enumerate}
          \item \textbf{Case $w=1$:}
            \begin{tasks}(2)
              \task[A.] \textbf{Case $\ell=w$.}
              \task[B.] \textbf{Case $\ell>w$.}
            \end{tasks}
          \item \textbf{Case $w=n-1$:}
          \begin{tasks}(2)
             \task[A.] \textbf{Case $\ell<w$.}
             \task[B.] \textbf{Case $\ell=w$.}
            \end{tasks}
          \item \textbf{Case $w\in\{1,\dots, n-1\}$:}
          \begin{tasks}(3)
              \task[A.] \textbf{Case $\ell<w$.}
              \task[B.] \textbf{Case $\ell=w$.}
              \task[C.] \textbf{Case $\ell>w$.}
            \end{tasks}
        \end{enumerate}
    \end{enumerate}
\end{enumerate}


We now explore each case in detail.

\begin{enumerate}[leftmargin=5.5mm]
    \item \textbf{Case $n=2$.}
    \begin{enumerate}[leftmargin=3mm]
        \item \textbf{Case $\ell\ge 2$.} By the scaling property of the exponential distribution, the first upward movement will have a length (in space) of distribution $\mbox{Exp}(\gamma/r(i))$, a jump to state $j$ at time $\thg_1$ occurs with probability $B_{ij}$, and then the subsequent downward movement will have a length (in space) of distribution $\mbox{Exp}(\gamma/|r(j)|)$. Thus, for $i\in\ME^+$ and $j\in\ME^-$,
    \begin{align*}
    \lambda^{(\ell,n)}_{ij}(s) & = B_{ij}\int_{0\vee s}^{\infty} \tfrac{\gamma}{r(i)}e^{-\tfrac{\gamma}{r(i)}u} \tfrac{\gamma}{|r(j)|}e^{-\tfrac{\gamma}{|r(j)|}(u-s)}\dd u\\
    & = B_{ij}\frac{\tfrac{\gamma}{r(i)}\tfrac{\gamma}{|r(j)|}}{\tfrac{\gamma}{r(i)}+\tfrac{\gamma}{|r(j)|}}e^{\tfrac{\gamma}{|r(j)|}s}e^{-\left(\tfrac{\gamma}{r(i)}+\tfrac{\gamma}{|r(j)|}\right) 0\vee s}\\
    & = \left\{\begin{array}{ccc}B_{ij} \frac{1}{r(i)+|r(j)|}\left(\gamma e^{-\tfrac{\gamma}{r(i)}s}\right)&\mbox{if}&s\ge 0,\\
    B_{ij} \frac{1}{r(i)+|r(j)|}\left(\gamma e^{\tfrac{\gamma}{|r(j)|}s}\right)&\mbox{if}&s< 0.
    \end{array}\right.
    \end{align*}
    In matrix terms, for $n=2$ we can write
    \begin{align}
    \bm{\Lambda}^{(\ell,n)}_{\ME^+\ME^-}(s)= \left\{\begin{array}{ccc}\bm{\Delta}_{\ME^+}(s)\,\bm{Q}^{(\ell,n)}_{\ME^+\ME^-}&\mbox{if}&s\ge 0,\\
    \bm{Q}^{(\ell,n)}_{\ME^+\ME^-}\bm{\Delta}_{\ME^-}(s)&\mbox{if}&s< 0,\end{array}\right.\label{eq:MatrixLambda1}
    \end{align}
    where for $n=2$ and $\ell\ge 2$,
    \begin{align*}
    \bm{Q}^{(\ell,n)}_{\ME^+\ME^-} & = \bm{B}_{\ME^+\ME^-}\odot \bm{H}_{\ME^+\ME^-},\\
    \bm{H}_{\ME^+\ME^-}& =\left\{\tfrac{1}{r(i) + |r(j)|}\right\}_{i\in\ME^+, j\in\ME^-}, \\ 
    \bm{\Delta}_{\ME^+}(s)& =\mbox{diag}\left\{\gamma e^{-\tfrac{\gamma}{r(i)}s}: i\in\ME^+\right\}, \\
    \bm{\Delta}_{\ME^-}(s)& =\mbox{diag}\left\{\gamma e^{\tfrac{\gamma}{|r(j)|}s}: j\in\ME^-\right\}.
    \end{align*}
    The key observation is that $\bm{\Lambda}^{(\ell,n)}_{\ME^+\ME^-}(s)$ can be written as a product of a matrix that only depends on the level $s$, and a matrix that depends only on the size of the bridge and the time of the switch. 
     \item \textbf{Case $\ell\le 0$.} In this case, the first upward movement will have a length (in space) of distribution $\mbox{Exp}(\gamma/\rho(i))$, a jump to state $j$ at time $\thg_1$ occurs with probability $B_{ij}$, and then the subsequent downward movement will have a length (in space) of distribution $\mbox{Exp}(\gamma/|\rho(j)|)$. Following the aforementioned algebraic steps, we arrive at the matrix expression:
    \begin{align}
    \bm{\Lambda}^{(\ell,n)}_{\MS^+\MS^-}(s)= \left\{\begin{array}{ccc}\bm{\Delta}_{\MS^+}(s)\,\bm{Q}^{(\ell,n)}_{\MS^+\MS^-}&\mbox{if}&s\ge 0,\\
    \bm{Q}^{(\ell,n)}_{\MS^+\MS^-}\bm{\Delta}_{\MS^-}(s)&\mbox{if}&s< 0,\end{array}\right.\label{eq:MatrixLambda3}
    \end{align}
    where for $n=2$ and $\ell\ge 2$,
    \begin{align*}
    \bm{Q}^{(\ell,n)}_{\MS^+\MS^-} & = \bm{B}_{\MS^+\MS^-}\odot \bm{H}_{\MS^+\MS^-},\\
    \bm{H}_{\MS^+\MS^-}& =\left\{\tfrac{1}{\rho(i) + |\rho(j)|}\right\}_{i\in\MS^+, j\in\MS^-}, \\ 
    \bm{\Delta}_{\MS^+}(s)& =\mbox{diag}\left\{\gamma e^{-\tfrac{\gamma}{\rho(i)}s}: i\in\MS^+\right\}, \\
    \bm{\Delta}_{\MS^-}(s)& =\mbox{diag}\left\{\gamma e^{\tfrac{\gamma}{|\rho(j)|}s}: j\in\MS^-\right\}.
    \end{align*}
    \item \textbf{Case $\ell=1$.} In this case, the first upward movement will have a length (in space) of distribution $\mbox{Exp}(\gamma/r(i))$, a jump to state $j$ at time $\thg_1$ occurs with probability $B_{ij}$, and then the subsequent downward movement will have a length (in space) of distribution $\mbox{Exp}(\gamma/|\rho(j)|)$. Following the aforementioned algebraic steps, we arrive at the matrix expression:
    \begin{align}
    \bm{\Lambda}^{(\ell,n)}_{\ME^+\MS^-}(s)= \left\{\begin{array}{ccc}\bm{\Delta}_{\ME^+}(s)\,\bm{Q}^{(\ell,n)}_{\ME^+\MS^-}&\mbox{if}&s\ge 0,\\
    \bm{Q}^{(\ell,n)}_{\ME^+\MS^-}\bm{\Delta}_{\MS^-}(s)&\mbox{if}&s< 0,\end{array}\right. \label{eq:MatrixLambda2}
    \end{align}
    where for $n=2$ and $\ell\ge 2$,
    \begin{align*}
    \bm{Q}^{(\ell,n)}_{\ME^+\MS^-} & = \bm{B}_{\ME^+\MS^-}\odot \bm{H}_{\ME^+\MS^-},\\
    \bm{H}_{\ME^+\MS^-}& =\left\{\tfrac{1}{r(i) + |\rho(j)|}\right\}_{i\in\ME^+, j\in\MS^-}, \\ 
    \bm{\Delta}_{\ME^+}(s)& =\mbox{diag}\left\{\gamma e^{-\tfrac{\gamma}{r(i)}s}: i\in\ME^+\right\}, \\
    \bm{\Delta}_{\MS^-}(s)& =\mbox{diag}\left\{\gamma e^{\tfrac{\gamma}{|\rho(j)|}s}: j\in\MS^-\right\}.
    \end{align*}
    \end{enumerate}
    It turns out that the decomposition shown in (\ref{eq:MatrixLambda1}), (\ref{eq:MatrixLambda2}) and (\ref{eq:MatrixLambda3}) holds for general $(\ell,n)$, which we verify recursively as we go along in our proof. As an induction step, suppose that these decompositions hold for all $(\ell',n')$  with $\ell'\ge 0$ and $n'\in \{1, 2, 3, \dots, n-1\}$.
    \item \textbf{Case $n\ge 3$.}
    \begin{enumerate}[leftmargin=3mm]
        \item \textbf{Case $\ell\ge n$.}
        In this case, the switching is inconsequential to the ruin of $\widetilde{F}$, whose underlying process $J$ evolves within $\ME$ in the time interval $[0,\thg_n)$.
        \begin{enumerate}[leftmargin=3mm]
        \item\textbf{Case $w=1$.} 
        Conditioning on the value of $\widetilde{F}(\thg_1)$, which is $\mbox{Exp}(\gamma/r(i))$-distributed, we get:
        \begin{align*}
           \Gamma^{(\ell,n,w)}_{ij}(s) = \int_{0\vee s} \tfrac{\gamma}{r(i)} e^{-\tfrac{\gamma}{r(i)} u} \sum_{i'\in\ME^+} B_{ii'} \lambda_{i'j}^{(\ell-1, n-1)}(s - u)\dd u.
        \end{align*}
        Employing the decomposition (\ref{eq:MatrixLambda1}) for the case $(\ell-1,n-1)$,
        \begin{align*}
            \Gamma^{(\ell,n,w)}_{ij}(s) & = \int_{0\vee s}^\infty \tfrac{\gamma}{r(i)} e^{-\tfrac{\gamma}{r(i)} u} \sum_{i'\in\ME^+} B_{ii'} Q_{i'j}^{(\ell-1, n-1)} (\gamma e^{\tfrac{\gamma}{|r(j)|}(s - u)})\dd u\\
            & =   \sum_{i'\in\ME^+} B_{ii'} Q_{i'j}^{(\ell-1, n-1)}|r(j)| \int_{0\vee s}^\infty \tfrac{\gamma}{r(i)} e^{-\tfrac{\gamma}{r(i)} u} \tfrac{\gamma}{|r(j)|} e^{-\tfrac{\gamma}{|r(j)|}(u-s)}\dd u\\
            &=\left\{\begin{array}{ccc} \sum_{i'\in\ME^+} B_{ii'} Q_{i'j}^{(\ell-1, n-1)}|r(j)| \left\{\frac{1}{r(i)+|r(j)|}\left(\gamma e^{-\tfrac{\gamma}{r(i)}s}\right)\right\}&\mbox{if}&s\ge 0,\\
            \sum_{i'\in\ME^+} B_{ii'} Q_{i'j}^{(\ell-1, n-1)}|r(j)| \left\{\frac{1}{r(i)+|r(j)|}\left(\gamma e^{\tfrac{\gamma}{|r(j)|}s}\right)\right\}&\mbox{if}&s< 0.
    \end{array}\right.
        \end{align*}
        In matrix terms, this can be expressed as 
        \begin{align*}
        \bm{\Gamma}^{(\ell,n,w)}_{\ME^+\ME^-}(s)= \left\{\begin{array}{ccc}\bm{\Delta}_{\ME^+}(s) \left\{\left(\bm{B}_{\ME^+\ME^+}\,\bm{Q}_{\ME^+\ME^-}^{(\ell-1, n-1)}\,\bm{R}_{\ME^-}\right)\odot \bm{H}_{\ME^+\ME-}\right\}&\mbox{if}&s \ge 0,\\
        \left\{\left(\bm{B}_{\ME^+\ME^+}\,\bm{Q}_{\ME^+\ME^-}^{(\ell-1, n-1)}\,\bm{R}_{\ME^-}\right)\odot \bm{H}_{\ME^+\ME^-}\right\} \bm{\Delta}_{\ME^-}(s)&\mbox{if}&s < 0,
        \end{array}\right.
            \end{align*}
            where
            \[\bm{R}_{\ME^-}=\mbox{diag}\left\{|r(j)|: j\in\ME^-\right\}.\]
             \item\textbf{Case $w=n-1$.} 
        Conditioning on the value of $\widetilde{F}(\thg_{n-1})-\widetilde{F}(\thg_n)$, which is $\mbox{Exp}(\gamma/|r(j)|)$-distributed, we get:
        \begin{align*}
           \Gamma^{(\ell,n,w)}_{ij}(s) = \int_{0\vee s}^\infty  \sum_{j'\in\ME^-} \lambda_{ij'}^{(\ell, n-1)}(u) B_{j'j} \left(\tfrac{\gamma}{|r(j)|} e^{-\tfrac{\gamma}{|r(j)|} (u-s)} \right)\dd u.
        \end{align*}
        Employing the decomposition (\ref{eq:MatrixLambda1}) for the case $(\ell,n-1)$,
        \begin{align*}
            \Gamma^{(\ell,n,w)}_{ij}(s) &= \int_{0\vee s}^\infty  \sum_{j'\in\ME^-} \left(\gamma e^{\tfrac{\gamma}{r(i)}s}\right) Q_{ij'}^{(\ell, n-1)}\, B_{j'j} \left(\tfrac{\gamma}{|r(j)|} e^{-\tfrac{\gamma}{|r(j)|} (u-s)} \right)\dd u\\
            & =  \sum_{j'\in\ME^-} r(i)\, Q_{ij'}^{(\ell, n-1)}\, B_{j'j} \int_{0\vee s}^\infty \tfrac{\gamma}{r(i)} e^{-\tfrac{\gamma}{r(i)} u} \tfrac{\gamma}{|r(j)|} e^{-\tfrac{\gamma}{|r(j)|}(u-s)}\dd u\\
            &=\left\{\begin{array}{ccc} \sum_{j'\in\ME^-} r(i)\, Q_{ij'}^{(\ell, n-1)}\, B_{j'j} \left\{\frac{1}{r(i)+|r(j)|}\left(\gamma e^{-\tfrac{\gamma}{r(i)}s}\right)\right\}&\mbox{if}&s\ge 0,\\
            \sum_{j'\in\ME^-} r(i)\, Q_{ij'}^{(\ell, n-1)}\, B_{j'j} \left\{\frac{1}{r(i)+|r(j)|}\left(\gamma e^{\tfrac{\gamma}{|r(j)|}s}\right)\right\}&\mbox{if}&s< 0.
    \end{array}\right.
        \end{align*}
        In matrix terms, this can be expressed as 
        \begin{align*}
        \bm{\Gamma}^{(\ell,n,w)}_{\ME^+\ME^-}(s)= \left\{\begin{array}{ccc}\bm{\Delta}_{\ME^+}(s) \left\{\left( \bm{R}_{\ME^+}\,\bm{Q}_{\ME^+\ME^-}^{(\ell, n-1)}\,\bm{B}_{\ME^-\ME^-}\right)\odot \bm{H}_{\ME^+\ME-}\right\}&\mbox{if}&s \ge 0,\\
        \left\{\left( \bm{R}_{\ME^+}\,\bm{Q}_{\ME^+\ME^-}^{(\ell, n-1)}\,\bm{B}_{\ME^-\ME^-}\right)\odot \bm{H}_{\ME^+\ME^-}\right\} \bm{\Delta}_{\ME^-}(s)&\mbox{if}&s < 0,
        \end{array}\right.
            \end{align*}
            where
            \[\bm{R}_{\ME^-}=\mbox{diag}\left\{|r(j)|: j\in\ME^-\right\}.\]
        \item\textbf{Case $w\in\{2,\dots, n-2\}$. }(Ignore this case if $n=3$) 
        Conditioning on the value of $\widetilde{F}(\thg_{w})$, we get:
        \begin{align*}
           \Gamma^{(\ell,n,w)}_{ij}(s) = \int_{0\vee s}^\infty  \sum_{j'\in\ME^-}\sum_{i'\in\ME^+} \lambda_{ij'}^{(\ell, w)}(u) B_{j'i'} \lambda_{i'j}^{(\ell-w, n-w)}(s-u) \dd u.
        \end{align*}
        Employing the decomposition (\ref{eq:MatrixLambda1}) for the cases $(\ell,w)$ and $(\ell-w,n-w)$,
        \begin{align*}
            & \Gamma^{(\ell,n,w)}_{ij}(s) = \int_{0\vee s}^\infty  \sum_{j'\in\ME^-}\sum_{i'\in\ME^+} \left(\gamma e^{\tfrac{\gamma}{r(i)}s}\right) Q_{ij'}^{(\ell, w)} B_{j'i'} Q_{i'j}^{(\ell-w, n-w)}\left(\gamma e^{\tfrac{\gamma}{|r(j)|}(s-u)}\right) \dd u\\
            & =  \sum_{j'\in\ME^-} \sum_{i'\in\ME^+} r(i)\,Q_{ij'}^{(\ell, w)} B_{j'i'} Q_{i'j}^{(\ell-w, n-w)}\,|r(j)| \int_{0\vee s}^\infty \tfrac{\gamma}{r(i)} e^{-\tfrac{\gamma}{r(i)} u} \tfrac{\gamma}{|r(j)|} e^{-\tfrac{\gamma}{|r(j)|}(u-s)}\dd u\\
            &=\left\{\begin{array}{ccc}  \sum_{j'\in\ME^-} \sum_{i'\in\ME^+} r(i)\,Q_{ij'}^{(\ell, w)} B_{j'i'} Q_{i'j}^{(\ell-w, n-w)}\,|r(j)| \left\{\frac{1}{r(i)+|r(j)|}\left(\gamma e^{-\tfrac{\gamma}{r(i)}s}\right)\right\}&\mbox{if}&s\ge 0,\\
             \sum_{j'\in\ME^-} \sum_{i'\in\ME^+} r(i)\,Q_{ij'}^{(\ell, w)} B_{j'i'} Q_{i'j}^{(\ell-w, n-w)}\,|r(j)| \left\{\frac{1}{r(i)+|r(j)|}\left(\gamma e^{\tfrac{\gamma}{|r(j)|}s}\right)\right\}&\mbox{if}&s< 0.
    \end{array}\right.
        \end{align*}
        In matrix terms, this can be expressed as 
        \begin{align*}
        \bm{\Gamma}^{(\ell,n,w)}_{\ME^+\ME^-}(s)= \left\{\begin{array}{ccc}\bm{\Delta}_{\ME^+}(s) \left\{\left( \bm{R}_{\ME^+}\,\bm{Q}_{\ME^+\ME^-}^{(\ell, w)}\,\bm{B}_{\ME^-\ME^+}\, \bm{Q}_{\ME^+\ME^-}^{(\ell-w, n-w)}\,\bm{R}_{\ME^-}\right)\odot \bm{H}_{\ME^+\ME-}\right\}&\mbox{if}&s \ge 0,\\
        \left\{\left( \bm{R}_{\ME^+}\,\bm{Q}_{\ME^+\ME^-}^{(\ell, w)}\,\bm{B}_{\ME^-\ME^+}\, \bm{Q}_{\ME^+\ME^-}^{(\ell-w, n-w)}\,\bm{R}_{\ME^-}\right)\odot \bm{H}_{\ME^+\ME-}\right\} \bm{\Delta}_{\ME^-}(s)&\mbox{if}&s < 0.
        \end{array}\right.
            \end{align*}
        \end{enumerate}
        \item \textbf{Case $\ell\le 0$.}
        In this case, the switching is inconsequential to the ruin of $\widetilde{F}$, whose underlying process $\widetilde{J}$ evolves within $\MS$ in the time interval $[0,\infty)$. The proof is very much the same as the case $\ell\ge n$, except that here we switch $\ME$ to $\MS$, and $r$ to $\rho$; for this reason, we just write the compact matrix form. In the following we allow $\ell$ to be negative, with the understanding that it coincides with the $\ell=0$ case.
        \begin{enumerate}[leftmargin=3mm]
        \item\textbf{Case $w=1$.} 
        \begin{align*}
        \bm{\Gamma}^{(\ell,n,w)}_{\MS^+\MS^-}(s)= \left\{\begin{array}{ccc}\bm{\Delta}_{\MS^+}(s) \left\{\left(\bm{B}_{\MS^+\MS^+}\,\bm{Q}_{\MS^+\MS^-}^{(\ell-1, n-1)}\,\bm{R}_{\MS^-}\right)\odot \bm{H}_{\MS^+\MS-}\right\}&\mbox{if}&s \ge 0,\\
        \left\{\left(\bm{B}_{\MS^+\MS^+}\,\bm{Q}_{\MS^+\MS^-}^{(\ell-1, n-1)}\,\bm{R}_{\MS^-}\right)\odot \bm{H}_{\MS^+\MS^-}\right\} \bm{\Delta}_{\MS^-}(s)&\mbox{if}&s < 0,
        \end{array}\right.
            \end{align*}
            where
            \[\bm{R}_{\MS^-}=\mbox{diag}\left\{|\rho(j)|: j\in\MS^-\right\}.\]
        \item\textbf{Case $w=n-1$.} 
        \begin{align*}
        \bm{\Gamma}^{(\ell,n,w)}_{\MS^+\MS^-}(s)= \left\{\begin{array}{ccc}\bm{\Delta}_{\MS^+}(s) \left\{\left( \bm{R}_{\MS^+}\,\bm{Q}_{\MS^+\MS^-}^{(\ell, n-1)}\,\bm{B}_{\MS^-\MS^-}\right)\odot \bm{H}_{\MS^+\MS-}\right\}&\mbox{if}&s \ge 0,\\
        \left\{\left( \bm{R}_{\MS^+}\,\bm{Q}_{\MS^+\MS^-}^{(\ell, n-1)}\,\bm{B}_{\MS^-\MS^-}\right)\odot \bm{H}_{\MS^+\MS^-}\right\} \bm{\Delta}_{\MS^-}(s)&\mbox{if}&s < 0,
        \end{array}\right.
            \end{align*}
            where
            \[\bm{R}_{\MS^-}=\mbox{diag}\left\{|\rho(j)|: j\in\MS^-\right\}.\]
        \item\textbf{Case $w\in\{2,\dots, n-2\}$.} (Ignore this case if $n=3$) 
        \begin{align*}
        \bm{\Gamma}^{(\ell,n,w)}_{\MS^+\MS^-}(s)= \left\{\begin{array}{ccc}\bm{\Delta}_{\MS^+}(s) \left\{\left( \bm{R}_{\MS^+}\,\bm{Q}_{\MS^+\MS^-}^{(\ell, w)}\,\bm{B}_{\MS^-\MS^+}\, \bm{Q}_{\MS^+\MS^-}^{(\ell-w, n-w)}\,\bm{R}_{\MS^-}\right)\odot \bm{H}_{\MS^+\MS-}\right\}&\mbox{if}&s \ge 0,\\
        \left\{\left( \bm{R}_{\MS^+}\,\bm{Q}_{\MS^+\MS^-}^{(\ell, w)}\,\bm{B}_{\MS^-\MS^+}\, \bm{Q}_{\MS^+\MS^-}^{(\ell-w, n-w)}\,\bm{R}_{\MS^-}\right)\odot \bm{H}_{\MS^+\MS-}\right\} \bm{\Delta}_{\MS^-}(s)&\mbox{if}&s < 0.
        \end{array}\right.
            \end{align*}
        \end{enumerate}
        \item \textbf{Case $\ell\in\{1,2,\dots, n-1\}$}. We proceed in a similar manner as in the previous cases by considering different cases for $w$.
        \begin{enumerate}[leftmargin=3mm]
        \item\textbf{Case $w=1$.} Here we have two further subcases.
        \begin{enumerate}[leftmargin=3mm]
            \item \textbf{Case $\ell=w$.} In this case, the first jump at $\thg_1$ causes a switch from $\ME$ to $\MS$. By analogous steps we get:
            \begin{align*}
            \bm{\Gamma}^{(\ell,n,w)}_{\ME^+\MS^-}(s)= \left\{\begin{array}{ccc}\bm{\Delta}_{\ME^+}(s) \left\{\left(\bm{B}_{\ME^+\MS^+}\,\bm{Q}_{\MS^+\MS^-}^{(\ell-1, n-1)}\,\bm{R}_{\MS^-}\right)\odot \bm{H}_{\ME^+\MS-}\right\}&\mbox{if}&s \ge 0,\\
            \left\{\left(\bm{B}_{\ME^+\MS^+}\,\bm{Q}_{\MS^+\MS^-}^{(\ell-1, n-1)}\,\bm{R}_{\MS^-}\right)\odot \bm{H}_{\ME^+\MS^-}\right\} \bm{\Delta}_{\MS^-}(s)&\mbox{if}&s < 0.
            \end{array}\right.
               \end{align*}
            \item \textbf{Case $\ell>w$.} In this case, the switch from $\ME$ to $\MS$ occurs $\ell-1$ steps away from $\thg_1$. Thus:
            \begin{align*}
            \bm{\Gamma}^{(\ell,n,w)}_{\ME^+\MS^-}(s)= \left\{\begin{array}{ccc}\bm{\Delta}_{\ME^+}(s) \left\{\left(\bm{B}_{\ME^+\ME^+}\,\bm{Q}_{\ME^+\MS^-}^{(\ell-1, n-1)}\,\bm{R}_{\MS^-}\right)\odot \bm{H}_{\ME^+\MS-}\right\}&\mbox{if}&s \ge 0,\\
            \left\{\left(\bm{B}_{\ME^+\ME^+}\,\bm{Q}_{\ME^+\MS^-}^{(\ell-1, n-1)}\,\bm{R}_{\MS^-}\right)\odot \bm{H}_{\ME^+\MS^-}\right\} \bm{\Delta}_{\MS^-}(s)&\mbox{if}&s < 0.
            \end{array}\right.
            \end{align*}
        \end{enumerate}
        \item\textbf{Case $w=n-1$.} Two cases as well
        \begin{enumerate}[leftmargin=3mm]
          \item \textbf{Case $\ell<w$.} 
            \begin{align*}
            \bm{\Gamma}^{(\ell,n,w)}_{\ME^+\MS^-}(s)= \left\{\begin{array}{ccc}\bm{\Delta}_{\ME^+}(s) \left\{\left( \bm{R}_{\ME^+}\,\bm{Q}_{\ME^+\MS^-}^{(\ell, n-1)}\,\bm{B}_{\MS^-\MS^-}\right)\odot \bm{H}_{\ME^+\MS-}\right\}&\mbox{if}&s \ge 0,\\
            \left\{\left( \bm{R}_{\ME^+}\,\bm{Q}_{\ME^+\MS^-}^{(\ell, n-1)}\,\bm{B}_{\MS^-\MS^-}\right)\odot \bm{H}_{\ME^+\MS^-}\right\} \bm{\Delta}_{\MS^-}(s)&\mbox{if}&s < 0,
            \end{array}\right.
            \end{align*}
            \item \textbf{Case $\ell=w$.} 
            \begin{align*}
            \bm{\Gamma}^{(\ell,n,w)}_{\ME^+\MS^-}(s)= \left\{\begin{array}{ccc}\bm{\Delta}_{\ME^+}(s) \left\{\left( \bm{R}_{\ME^+}\,\bm{Q}_{\ME^+\ME^-}^{(\ell, n-1)}\,\bm{B}_{\ME^-\MS^-}\right)\odot \bm{H}_{\ME^+\MS-}\right\}&\mbox{if}&s \ge 0,\\
            \left\{\left( \bm{R}_{\ME^+}\,\bm{Q}_{\ME^+\ME^-}^{(\ell, n-1)}\,\bm{B}_{\ME^-\MS^-}\right)\odot \bm{H}_{\ME^+\MS^-}\right\} \bm{\Delta}_{\MS^-}(s)&\mbox{if}&s < 0,
            \end{array}\right.
            \end{align*}
            \end{enumerate}

        \item\textbf{Case $w\in\{2,\dots, n-2\}$.} (Ignore this case if $n=3$) Here we have three cases.
        \begin{enumerate}[leftmargin=3mm]
            \item \textbf{Case $\ell<w$}
            \begin{align*}
            \bm{\Gamma}^{(\ell,n,w)}_{\ME^+\MS^-}(s)= \left\{\begin{array}{ccc}\bm{\Delta}_{\ME^+}(s) \left\{\left( \bm{R}_{\ME^+}\,\bm{Q}_{\ME^+\MS^-}^{(\ell, w)}\,\bm{B}_{\MS^-\MS^+}\, \bm{Q}_{\MS^+\MS^-}^{(\ell-w, n-w)}\,\bm{R}_{\MS^-}\right)\odot \bm{H}_{\ME^+\MS-}\right\}&\mbox{if}&s \ge 0,\\
            \left\{\left( \bm{R}_{\ME^+}\,\bm{Q}_{\ME^+\MS^-}^{(\ell, w)}\,\bm{B}_{\MS^-\MS^+}\, \bm{Q}_{\MS^+\MS^-}^{(\ell-w, n-w)}\,\bm{R}_{\MS^-}\right)\odot \bm{H}_{\ME^+\MS-}\right\} \bm{\Delta}_{\MS^-}(s)&\mbox{if}&s < 0.
            \end{array}\right.
            \end{align*}
            \item \textbf{Case $\ell=w$}
            \begin{align*}
            \bm{\Gamma}^{(\ell,n,w)}_{\ME^+\MS^-}(s)= \left\{\begin{array}{ccc}\bm{\Delta}_{\ME^+}(s) \left\{\left( \bm{R}_{\ME^+}\,\bm{Q}_{\ME^+\ME^-}^{(\ell, w)}\,\bm{B}_{\ME^-\MS^+}\, \bm{Q}_{\MS^+\MS^-}^{(\ell-w, n-w)}\,\bm{R}_{\MS^-}\right)\odot \bm{H}_{\ME^+\MS-}\right\}&\mbox{if}&s \ge 0,\\
            \left\{\left( \bm{R}_{\ME^+}\,\bm{Q}_{\ME^+\ME^-}^{(\ell, w)}\,\bm{B}_{\ME^-\MS^+}\, \bm{Q}_{\MS^+\MS^-}^{(\ell-w, n-w)}\,\bm{R}_{\MS^-}\right)\odot \bm{H}_{\ME^+\MS-}\right\} \bm{\Delta}_{\MS^-}(s)&\mbox{if}&s < 0.
            \end{array}\right.
            \end{align*}
            \item \textbf{Case $\ell>w$}
            \begin{align*}
            \bm{\Gamma}^{(\ell,n,w)}_{\ME^+\MS^-}(s)= \left\{\begin{array}{ccc}\bm{\Delta}_{\ME^+}(s) \left\{\left( \bm{R}_{\ME^+}\,\bm{Q}_{\ME^+\ME^-}^{(\ell, w)}\,\bm{B}_{\ME^-\ME^+}\, \bm{Q}_{\ME^+\MS^-}^{(\ell-w, n-w)}\,\bm{R}_{\MS^-}\right)\odot \bm{H}_{\ME^+\MS-}\right\}&\mbox{if}&s \ge 0,\\
            \left\{\left( \bm{R}_{\ME^+}\,\bm{Q}_{\ME^+\ME^-}^{(\ell, w)}\,\bm{B}_{\ME^-\ME^+}\, \bm{Q}_{\ME^+\MS^-}^{(\ell-w, n-w)}\,\bm{R}_{\MS^-}\right)\odot \bm{H}_{\ME^+\MS-}\right\} \bm{\Delta}_{\MS^-}(s)&\mbox{if}&s < 0.
            \end{array}\right.
            \end{align*}
        \end{enumerate}
        \end{enumerate}
    \end{enumerate}
\end{enumerate}

Employing (\ref{eq:LambdasumGamma1}), (\ref{eq:LambdasumGamma2}) and (\ref{eq:LambdasumGamma3}), all the cases for $s<0$ can be summarized in matrix form:
\begin{align}
\begin{pmatrix}
\bm{\Lambda}^{(\ell,n)}_{\ME^+\ME^-}(s) & \bm{\Lambda}^{(\ell,n)}_{\ME^+\MS^-}(s)\\ \bm{0} & \bm{\Lambda}^{(\ell,n)}_{\MS^+\MS^-}(s)
\end{pmatrix} = \bm{Q}^{(\ell,n)}\begin{pmatrix}\bm{\Delta}_{\ME^-}(s) & \bm{0}\\\bm{0}&\bm{\Delta}_{\MS^-}(s)\label{eq:decomposition1}
\end{pmatrix},
\end{align}
where $\bm{Q}^{(\ell,n)}$ is the block matrix defined recursively in the statement of Theorem \ref{th:main}. Finally, 
\[\bm{\Psi}^{(\ell, n)} = \bm{Q}^{(\ell,n)} \left(\int_{-\infty}^0 \begin{pmatrix}\bm{\Delta}_{\ME^-}(s) & \bm{0}\\\bm{0}&\bm{\Delta}_{\MS^-}(s)
\end{pmatrix} \dd s \right) = \bm{Q}^{(\ell,n)} \bm{R}_-,\]
which completes the proof.

\begin{Remark}Equation \eqref{eq:decomposition1} reveals a unique property of the density associated with the $n$-bridges: the structure of the $n$-bridge, encapsulated by the matrix $\bm{Q}^{(\ell,n)}$, is \emph{independent} of the level that this $n$-bridge reaches at its right endpoint, which is of an exponentially-distributed nature. Although this might initially seem counterintuitive, it can be heuristically explained by the notion that the endpoints of an $n$-bridge are competing exponential random variables (in opposite directions). Therefore, once the underlying states of the endpoints are established, the actual level reached by the right endpoint simply follows a two-sided exponential distribution. This striking property illuminates further applications of $n$-bridges.
\end{Remark}

\section{Concluding remarks and further applications}\label{sec:conclusion}

We proposed a model for bivariate stochastic fluid processes that incorporates a ruin-dependent behavioral switch. We have introduced a class of approximations, called compatible pastings. This allows for the development of a tractable bivariate model that retains the essential characteristics of the original process and whose first passage probabilities can be computed efficiently. In particular, the behavioral switch of the compatible pasting process can be observed on a high-frequency grid. We have analyzed the first passage problem using the concept of $n$-bridges. These are particular shapes of paths with given number of observations, and which stand above their end points. Importantly, the observation of the behavioral switch allows us to further condition the n-bridges on this Poissonian point. 

Our model is relevant to a diverse range of real-world systems. For instance, it is pertinent to queuing theory, with potential use cases spanning scenarios from telecommunications (where data flow can be disrupted by server downtime) to supply chain management (where sudden resource shortages can provoke operational shifts). Furthermore, while our study primarily concentrated on ruin, the methodology we employed can be extended to calculate other functionals, such as Laplace transform of the lifetime reward from the fluid process.
Moreover, the behavioral switch can occur according to other criteria than ruin. For example, when one of the processes reaches a certain barrier, or a Parisian ruin. 

The focus of this work has been on the scenario where the initial value of the fluid process is $0$. However, circumstances with initial capital greater than $0$ can also be explored by employing Erlangization and fluidization techniques; see e.g., \cite[Section 6]{amini2023duration}. Another interesting extension of this work would be to integrate external controls into the system. In this modification, the behavioral switch would not solely depend on the ruin of another coordinate but could also be triggered by an external agent. This enhancement could allow for the creation of more flexible and responsive systems, wherein changes are not purely reactive but can also be preemptive, offering control over the system's dynamics.

Our analysis of first return probabilities can help understand the implications of the ruin-dependent behavioral switch for the performance and stability of the system. The proposed bivariate stochastic fluid process model with ruin-dependent behavioral switching is a versatile  tool for modeling and analyzing  interdependent systems. Moreover, by incorporating  behavioral switches  that depend, for example, on the proportion of ruined agents, we could extend our model to more complex settings in multiple dimensions.

A further research direction is the investigation of sensitivity analysis with respect to system parameters. This could include exploring how different parameters change the distribution of ruin times. This analysis would inform decision-makers and lead to better design and control of such interdependent systems.
\paragraph{Acknowledgements.} We gratefully acknowledge financial support from NSF Award \#1653354 and AXA Research Fund Award on `Mitigating risk in the wake of the pandemic'.

\bibliographystyle{abbrv}
\bibliography{oscar}
\end{document}